\documentclass[oneside,12pt]{article}

\usepackage{amsmath}
\usepackage{amssymb}
\usepackage{dsfont}
\usepackage{pxfonts}
\usepackage{graphicx}
\usepackage{eucal}
\usepackage{mathrsfs}
\usepackage{theorem}
\usepackage{pifont}
\usepackage{sectsty}
\usepackage{amscd}
\usepackage{color}
\usepackage{fancyhdr}
\usepackage{framed}
\usepackage[all]{xy}
\usepackage[pagebackref]{hyperref}

\usepackage{tikz} 
\usepackage{tkz-euclide}
\usetkzobj{all}
\usepackage{pgfplots}

\newcommand{\btkz}{\begin{tikzpicture}}
\newcommand{\etkz}{\end{tikzpicture}}

\definecolor{shadecolor}{rgb}{0.8,0.8,0.8}

\theoremheaderfont{\fontfamily{pzc}\bfseries\large}
\newtheorem{theorem}{Theorem}[section]

\newtheorem{lemma}[theorem]{Lemma}
\newtheorem{proposition}[theorem]{Proposition}

\newtheorem{corollary}[theorem]{Corollary}
\newtheorem{definition}[theorem]{Definition}

{\theorembodyfont{\rmfamily}

\newcommand{\specexercise}[1]{}

\newenvironment{proof}{{\flushleft \emph{Proof}:}}{\hfill\ding{110}}

\newenvironment{remark}{{\flushleft \fontfamily{pzc}\bfseries\large Remark:}}{}

\newcommand{\g}{\mathfrak{g}}
\renewcommand{\a}{\mathfrak{a}}
\renewcommand{\b}{\mathfrak{b}}
\renewcommand{\c}{\mathfrak{c}}
\newcommand{\euc}{\mathfrak{e}}
\newcommand{\e}{\varepsilon}
\newcommand{\Volume}{d\text{Vol}_\g}

\newcommand{\R}{{\mathbb{R}}}
\newcommand{\M}{{\mathcal{M}}}

\newcommand{\E}{{\mathcal{E}}}
\newcommand{\Ebond}{{\E_\e^{\text{Bond}}}}
\newcommand{\Evol}{{\E_\e^{\text{Vol}}}}
\newcommand{\Wbond}{{W_\e^{\text{Bond}}}}
\newcommand{\Wvol}{{W_\e^{\text{Vol}}}}
\newcommand{\Wtot}{{W_\e^{\text{Total}}}}

\newcommand{\scrO}{\mathscr{O}}

\newcommand{\fraku}{\mathfrak{u}}
\newcommand{\calX}{{\mathcal{X}}}

\newcommand{\Vol}{\operatorname{Vol}}
\newcommand{\dist}{\operatorname{dist}}
\newcommand{\Hom}{\operatorname{Hom}}
\newcommand{\SO}[1]{\text{SO}(#1)}
\renewcommand{\O}[1]{\text{O}(#1)}

\newcommand{\Wb}{\Phi}

\newcommand{\alPhi}{\alpha_\Phi}
\newcommand{\CPhi}{C_\Phi}
\newcommand{\LPhi}{L_\Phi}
\newcommand{\alPsi}{\alpha_\Psi}
\newcommand{\CPsi}{C_\Psi}
\newcommand{\LPsi}{L_\Psi}
\newcommand{\alW}{\alpha_W}
\newcommand{\CW}{C_W}

\newcommand{\pl}{\partial}

\newcommand{\expNAB}{\exp^\nabla}

\newcommand{\id}{\operatorname{Id}}
\newcommand{\ind}{\mathds{1}}

\newcommand{\Cext}{C_{\text{ext}}}

\newcommand{\textand}{\quad\text{ and }\quad}
\newcommand{\Textand}{\qquad\text{ and }\qquad}

\def\XXint#1#2#3{{\setbox0=\hbox{$#1{#2#3}{\int}$ }
\vcenter{\hbox{$#2#3$ }}\kern-.6\wd0}}

\newcommand{\brk}[1]{\left(#1\right)}          
\newcommand{\Brk}[1]{\left[#1\right]}          
\newcommand{\BRK}[1]{\left\{#1\right\}}        
\newcommand{\Abs}[1]{\left| #1 \right|}        
\newcommand{\Cases}[1]{\begin{cases} #1 \end{cases}}

\newcommand{\deriv}[2]{\frac{d#1}{d#2}}

\numberwithin{equation}{section}

\begin{document}

\title{Variational Convergence of Discrete Geometrically-Incompatible Elastic Models}
\author{Raz Kupferman\footnote{Institute of Mathematics, The Hebrew University.}\, and Cy Maor\footnote{Department of Mathematics, University of Toronto.}}
\date{}
\maketitle

\begin{abstract}
We derive a continuum model for incompatible elasticity as a variational limit of a family of discrete nearest-neighbor elastic models. The discrete models are based on discretizations of a smooth Riemannian manifold $(\M,\g)$, endowed with a flat, symmetric connection $\nabla$. The metric $\g$ determines local equilibrium distances between neighboring points; the connection $\nabla$ induces a lattice structure shared by all the discrete models. The limit model satisfies a fundamental rigidity property: there are no stress-free configurations, unless $\g$ is flat, i.e., has zero Riemann curvature. Our analysis focuses on two-dimensional systems, however, all our results readily generalize to higher dimensions.
\end{abstract}

\tableofcontents

\section{Introduction}
Incompatible, or non-Euclidean elasticity is a model of pre-stressed materials, in which
the elastic body is modeled as an oriented Riemannian manifold with boundary $(\M^d,\g)$ (usually $d=2,3$). The metric $\g$, often called a \emph{reference metric}, represents local equilibrium distances between neighboring material elements.
The elastic energy of a configuration $f:\M\to \R^d$ is a measure of distortion; it quantifies how far $f$ is from being an isometry. If we denote the (non-negative) elastic energy density by $W$, then
\begin{equation}
\label{eq:energy_density_basic_assumption}
W(df_p) = 0 
\qquad\text{if and only if}\qquad
df_p\in \SO{\g_p,\euc},
\end{equation}
where $p\in\M$, $\euc$ is the Euclidean metric and $\SO{\g_p,\euc}$ is the set of orientation-preserving isometries $T_p\M\to\R^d$.
If the curvature tensor of $\g$ is not identically zero, then there are no isometric immersions $\M\to \R^d$. Thus, the elastic energy associated with every configuration is positive even in the absence of external forces or boundary conditions. In other words, there are \emph{no stress-free configurations} (and under some mild coercivity conditions on $W$, the infimum energy is also positive \cite{LP10,KMS17}). 
This is in contrast to classical elasticity, where the existence of a stress-free configuration is assumed explicitly---the \emph{reference configuration}---which amounts to assuming that $\g$ is flat.

The study of non-Euclidean elasticity was initiated in the 1950s \cite{Nye53,Kon55,BBS55}, followed later by \cite{Wan67,Kro81} and others. In these works, incompatibility is due to material defects, such as disclinations, dislocations and point-defects.
In recent years, the scope of incompatible elasticity has been extended significantly, encompassing differential growth \cite{GB05,Yav10}, humidity-driven expansion and shrinkage \cite{AESK11}, thermal expansion \cite{OY10}, responsive gels \cite{KES07} and more.

A central theme in continuum mechanics is the derivation of continuum models from discrete particle models. 
A prototypical example is the modeling of an elastic medium by a collection of point masses interconnected by (possibly nonlinear) springs. 
One then analyzes the convergence of the discrete model as the spacing between adjacent particles tends to zero. 
There is a wealth of work addressing the discrete-to-continuum limit, both in the context of crystalline structures and disordered, or amorphous media (see e.g., references in \cite[p.~84]{Bra02}; for a more recent treatment see e.g., \cite{LR13} and the references therein).
However, to the best of our knowledge, there are no such results in the context of non-Euclidean elasticity, in which the limiting model satisfies Property \eqref{eq:energy_density_basic_assumption}.
The goal of this paper is to derive a continuum model for non-Euclidean elasticity as a limit of a certain class of discrete models, and to provide a framework for studying the limit of other classes of discrete models. 

The derivation of a continuum model as a limit of discrete models requires a well-defined limiting process, that is, a family of discrete models parametrized by decreasing inter-particle separations.
In a Euclidean body, this construction is natural---a global lattice structure can be defined by  Cartesian coordinates, yielding a family of lattices (e.g., hexagonal or cubic)  of varying scale.
In a manifold endowed with a non-Euclidean metric, there are no canonical choices of lattices;
supplementary geometric information must be prescribed in order to specify a lattice structure. 

A natural construct, which the body manifold can be endowed with, and through which a parametrized family of lattices can be defined is a flat affine connection $\nabla$. 
In general, an affine connection defines parallel transport of tangent vectors along paths. If the connection is flat (and the manifold is simply-connected), parallel transport is path-independent, and in particular, there exist 
parallel frame fields. In the context of this paper, parallel frame fields represent the underlying lattice directions, or equivalently, they define how lattice directions transform between different points in the body.
In materials science, an affine connection models distributions of defects; the defect density is represented by the tensorial fields---\emph{curvature}, \emph{torsion} and \emph{non-metricity}---that define the connection.
For a more detailed discussion on defects and affine connections see the classical papers \cite{Wan67,Kro81}, or more recently \cite{YG12,KM15,KMR16}.

\paragraph{Description of main results}
In this paper, we derive a continuum model for a residually-stressed elastic body as a variational limit of discrete models;
the limit functional satisfies Property \eqref{eq:energy_density_basic_assumption}, hence the limit model does not admit a reference configuration unless the reference metric is flat.
For the sake of clarity, we present results in a two-dimensional setting; higher dimensional generalizations are discussed at the end of the paper.

The discrete model is 
based on a flat and symmetric connection $\nabla$ and an underlying hexagonal lattice.
The lattice is defined by three crystallographic axes $\{\a,\b,\c\}$, which are $\nabla$-parallel vector fields spanning $T\M$, satisfying $\a+\b+\c=0$. For every $\e>0$, $\M$ is triangulated such that the distance between every two adjacent vertices is of order $\e$ and the $\nabla$-geodesic connecting them is parallel to a crystallographic axis. As we show, the flatness and the symmetry of $\nabla$ make such a triangulation possible.

Given $\e>0$, denote by $V_\e$ the vertices of the triangulation. A discrete configuration is a mapping $f_\e:V_\e\to\R^2$. To every discrete configuration corresponds an energy comprising
 two nearest-neighbor contributions: one contribution is a pairwise bond term, penalizing incompatible distances between two adjacent vertices; the other contribution is a three-point volume term, penalizing incompatibility in the signed volume of triangles.
The volume term is natural both from physical and analytic perspectives, as it penalizes for orientation-reversal, or ``folding".
From a physical perspective, the volume term implies an energy cost for material interpenetration; 
from an analytic perspective, a nearest-neighbor model indifferent to folding cannot yield Property \eqref{eq:energy_density_basic_assumption} in the $\e\to0$ limit.

Our main results are (i) the family of discrete energies $\Gamma$-converges, as $\e\to0$, in a strong $L^2$ topology (Theorem~\ref{thm:Gamma_convergence}); to this end we embed discrete configurations into $L^2(\M;\R^2)$ via a standard affine extension.
(ii) The limit functional satisfies Property \eqref{eq:energy_density_basic_assumption} and frame-invariance (left $\text{SO}$-symmetry) (Proposition~\ref{pn:properties_limit_functional}); an additional discrete right symmetry and a material homogeneity property are obtained under additional assumptions (Section~\ref{sec:conformal}).

\paragraph{Comparison with previous results}
To the best of our knowledge, the only existing work on a discrete-to-continuum limit in non-Euclidean elasticity is the work of Lewicka and Ochoa \cite{LO15}.
In \cite{LO15}, a family of discretizations is constructed by designating a distinguished coordinate chart and constructing cubic lattices in these coordinates.
This construction is equivalent to a choice of a flat, symmetric connection $\nabla$, so while differing in terminology, the construction in \cite{LO15} is similar to ours.

The limiting model obtained in \cite{LO15} has one major drawback: Property \eqref{eq:energy_density_basic_assumption} is not satisfied; this follows immediately from \cite[Lemma~4.1]{LO15}, and manifests in \cite[Example~6.3]{LO15}. This drawback is due to the fact that  the energy only accounts for pairwise bond energies in a cubic lattice. 
A nearest-neighbor model in a cubic lattice does not penalize shear and as a result, neither does the limiting continuum model.  
Moreover, pairwise bond energy does not penalize for orientation-reversal, or folding; as a result, the convexification occurring in the limit process yields a limit energy indifferent to contractions.

These problems may disappear if one considers non-nearest-neighbor interactions; however, these cases are much harder to analyze. Non-nearest-neighbor interactions were considered in \cite{LO15}. They obtained bounds for the limit energies ($\Gamma-\liminf$ and $\Gamma-\limsup$ inequalities); the lower bounds, however,  do not satisfy Property \eqref{eq:energy_density_basic_assumption}.

\paragraph{Physical interpretation of the model and future directions}
In this paper, we derive the limit energy of an incompatible elastic body endowed with a flat, symmetric connection. 
Such a model may be relevant to several physical settings. One such setting is that of a body undergoing differential expansion (or shrinkage); initially the metric is Euclidean and the body is endowed with a lattice structure conforming with that metric. After expansion, the body acquires a new reference metric---a non-Euclidean one---while retaining the original lattice structure.
This setting is a lattice-equivalent of the settings considered in various recent experimental and theoretical works \cite{ESK09a,OY10}.
Another type of systems modeled by flat symmetric connections consists of bodies containing distributed point-defects; see \cite{KMR16}.

Mathematically, flat symmetric connections are the easiest to handle: their specification is equivalent to choosing a distinguished coordinate chart $x=(x^i)$ on $\M$, and declaring the frame fields $\pl_{x^i}$ to be parallel.
Our choice of working with a connection rather than with a distinguished coordinate chart emphasizes the geometric and the physical meaning of the underlying lattice structure; in particular, some of the properties of the limit energy, e.g., material homogeneity (existence of a material connection), are more naturally formulated and derived in this framework.

All the above-mentioned models assume an underlying lattice structure; in particular, the limit models are not isotropic.
Another important future direction is obtaining the isotropic elastic energy of an amorphous pre-stressed elastic body.
A possible approach is using a random discretization of the manifold, similar to the work of \cite{ACG11} in the Euclidean case.

\paragraph{Structure of the paper}
In Section~\ref{sec:prelim} we derive key properties of flat symmetric connections and define the affine extension of functions defined on the vertices of geodesic triangles.
In Section~\ref{sec:model} we define the discrete elastic models.
In Section~\ref{sec:Gamma} we prove the variational convergence of the discrete elastic models.
In Section~\ref{sec:prop_of_limit} we prove key properties of the limit functional, and notably Property~\eqref{eq:energy_density_basic_assumption}. 
In Section~\ref{sec:dicussion} we discuss various possible extensions of this work, including higher dimensions, different lattice structures
and models forbidding interpenetration.

\section{Flat symmetric connections and affine extensions}
\label{sec:prelim}

Throughout this paper, $\M$ is assumed to be a smooth, simply-connected manifold, possibly with a smooth boundary.
Let $\nabla$ be a flat connection on $\M$. As is well-known,  the parallel transport induced by a flat connection on a simply-connected manifold is path-independent.
We denote the parallel transport operator by
\[
\Pi_p^q : T_p\M\to T_q\M.
\]
For $p\in\M$, we denote by $\expNAB_p:\scrO_p\subset T_p\M\to\M$ the exponential map at $p$, where $\scrO_p$ is a convex neighborhood of the origin in $T_p\M$, such that $\expNAB_p$ is a smooth embedding.

We start this section by establishing a number of properties pertinent to flat, symmetric connections:

\begin{lemma}
\label{lemma:additive_exp}
Let $\M$ be a simply-connected $d$-dimensional manifold, $d\ge2$, and let $\nabla$ be a flat, symmetric connection on $\M$. Let $p\in\M$ and let $v,w\in \scrO_p$ be independent vectors, such that $v+w\in\scrO_p$.
Then,
\[
\expNAB_p(v+w) = \expNAB_{\expNAB_pv}(\Pi_p^{\expNAB_pv}w).
\]  
\end{lemma}

\begin{proof}
Define the parallel vector fields $X,Y\in\Gamma(T\M)$ given by
\[
X_q = \Pi_p^q v
\Textand
Y_q = \Pi_p^q w.
\]
Since $X$ and $Y$ are parallel and since the connection is symmetric, it follows that $[X,Y]=0$.
It is well-known that the flows induced by commuting flow fields satisfy the additive relation $\phi_{X+Y}^t  = \phi_Y^t\circ \phi_X^t$ (follows from \cite[Prop.~18.5]{Lee06}).
Finally, the flow induced by a parallel frame field is related to the exponential map via 
\[
\phi_X^t(p) =  \expNAB_p(t X_p).
\]
Thus,
\[
\phi_{X+Y}^1(p) = \expNAB_p(X_p + Y_p) = \expNAB_p(v + w),
\]
and
\[
\phi_Y^1\circ \phi_X^1(p) = \phi_Y^1\brk{\expNAB_p v} = \expNAB_{\expNAB_p v}(Y_{\expNAB_p v}) = 
 \expNAB_{\expNAB_p v}(\Pi_p^{\expNAB_p v}w).
\] 
\end{proof}

The following corollary results immediately from Lemma~\ref{lemma:additive_exp}, applied to the segment $t \mapsto v + tw$.
\begin{corollary}
\label{cor:lines_are_geodesics}
For $p\in\M$,
the map $\expNAB_p$ maps straight segments in $\scrO_p$ into $\nabla$-geodesics (note that for general connections, this is only true for lines through the origin). 
\end{corollary}

%

\begin{lemma}
\label{lemma:dexp}
Let the setting by the same as in Lemma~\ref{lemma:additive_exp}. Then,
\[
(d\expNAB_p)_v = \Pi_p^{\expNAB_p v}.
\]
\end{lemma}

\begin{proof}
By definition, for $v,w\in T_p\M$,
\[
\begin{split}
(d\expNAB_p)_v(w) &= \left.\deriv{}{t}\right|_{t=0} \expNAB_p(v + tw) 
= \left.\deriv{}{t}\right|_{t=0} \expNAB_{\expNAB_p v}(t \, \Pi_p^{\expNAB_pv} w) \\
&= \brk{d\expNAB_{\expNAB_p v}}_0\brk{\Pi_p^{\expNAB_p v} w} 
= \Pi_p^{\expNAB_pv} w.
\end{split}
\]
The first and the third equalities follow from the definition of the differential.
In the second equality we used Lemma~\ref{lemma:additive_exp}; in the last one we used the fact that the differential of the exponential map at the origin is the identity map.
\end{proof}

From now on we focus on two-dimensional manifolds endowed with a flat, symmetric connection $\nabla$.

\begin{corollary}
\label{cor:2.4}
Let $p\in\M$ and 
\[
q = \expNAB_p v
\Textand
r = \expNAB_p w,
\]
for $v,w\in T_p\M$.
Then, 
\[
r = \expNAB_q(\Pi_p^q(w-v)).
\]
Moreover, the interior of the $\nabla$-geodesic triangle whose vertices are $p$, $q$ and $r$ is given by
\[
\mathcal{T}_\nabla(p,q,r) = \{\expNAB_p(sv  + tw) ~:~ 0< s,t, s+t < 1\}.
\]
In other words, $\mathcal{T}_\nabla(p,q,r)$ is the image under $\expNAB_p$ of the triangle
\[
\mathcal{T}_p(0,v,w) \subset T_p\M,
\]
whose vertices are $0$, $v$ and $w$.
\end{corollary}

\begin{proof}
If follows from Lemma~\ref{lemma:additive_exp} that
\[
\begin{split}
\expNAB_q(\Pi_p^q(w-v)) &= \expNAB_{\expNAB_p v}(\Pi_p^{\expNAB_p v}(w-v)) 
= \expNAB_{p}(v + (w-v)) = r.
\end{split}
\]
The fact that the triangle $pqr$ is the image of $\mathcal{T}_p(0,v,w)$ is an immediate consequence of the previous corollary, whereby straight lines in $T_p\M$ are mapped into $\nabla$-geodesics; in particular, the curve $t\mapsto\expNAB_p((1-t)v+tw)$, $t\in[0,1]$ is the geodesic between $q$ and $r$.
\end{proof}

Let $p$, $q$ and $r$ be defined as in Corollary~\ref{cor:2.4}.
Suppose that the values of a real-valued function $f$ are prescribed at the points $p,q,r$. We define the extension $F$ of $f$ in the geodesic triangle $\mathcal{T}_\nabla(p,q,r)$ as follows:
\[
F(\expNAB_p(s v + t w)) = f(p) + s\, (f(q) - f(p)) + t\,(f(r) - f(p)),
\qquad 0\le s,t,s+t \le 1.
\]
It immediately follows from the definition that the function $F\circ\expNAB_p$ is a real-valued affine function on $\mathcal{T}_p(0,v,w)$:
\begin{proposition}
\label{prop:dF1}
Let $f$ and $F$ be defined as above.
The differential of $F$ satisfies
\[
d(F\circ\expNAB_p)_{sv + tw}(v) = f(q) - f(p)
\Textand
d(F\circ\expNAB_p)_{sv + tw}(w) = f(r) - f(p).
\]
\end{proposition}

%

\begin{proposition}
\label{prop:dF2}
Let $f$ and $F$ be defined as above.
$F$ is affine in the following sense: for every $x\in \mathcal{T}_\nabla(p,q,r)$, 
\[
dF_x(\Pi_p^x v) = f(q) - f(p)
\Textand
dF_x(\Pi_p^x w) = f(r) - f(p).
\]
\end{proposition}

\begin{proof}
By definition,
\[
dF_x(\Pi_p^x v) =  \left.\deriv{}{\tau}\right|_{\tau=0} F\circ\gamma(\tau),
\]
where $\gamma(0) = x$ and $\dot{\gamma}(0) = \Pi_p^x v$. Taking
\[
\gamma(\tau) = \expNAB_x(\tau\, \Pi_p^x v),
\]
we obtain
\[
\begin{split}
dF_x(\Pi_p^x v) &=  \left.\deriv{}{\tau}\right|_{\tau=0} F\circ \expNAB_x(\tau\, \Pi_p^x v).
\end{split}
\]
Let $x = \expNAB_p(sv  +tw)$. Then, using once again Lemma~\ref{lemma:additive_exp},
\[
\begin{split}
dF_x(\Pi_p^x v) &=  \left.\deriv{}{\tau}\right|_{\tau=0} F\circ \expNAB_{\expNAB_p(sv  +tw)}(\tau\, \Pi_p^{\expNAB_p(sv  +tw)} v) \\
&=  \left.\deriv{}{\tau}\right|_{\tau=0} F\circ \expNAB_{p}(sv  + tw + \tau  v) \\
&= d(F\circ\expNAB_p)_{sv + tw}(v) \\
&=  f(q) - f(p),
\end{split}
\]
where the last equality follows from Proposition~\ref{prop:dF1}.
The other statements are analogous.
\end{proof}

\section{Discrete elastic model}
\label{sec:model}

Let $(\M,\g)$ be a compact two-dimensional Riemannian manifold; for simplicity, we assume $\M$ to be a topological disc with a smooth boundary.
Let $\nabla$ be a flat, symmetric connection on $\M$. 

Let $\{\a,\b,\c\}$ be three $\nabla$-parallel vector fields, satisfying
\[
\a + \b + \c = 0,
\] 
such that every two constitute a parallel frame field. That is, $\a$, $\b$ and $\c$ are nowhere co-linear and for every $p,q\in\M$,
\[
\a_p = \Pi_q^p \a_q,
\quad
\b_p = \Pi_q^p \b_q
\Textand
\c_p = \Pi_q^p \c_q.
\]
The vector fields $\a$, $\b$ and $\c$ represent the three crystallographic axes of an hexagonal lattice.
Note that this structure is independent of any metric properties. 

\subsection{Metric}

We endow $\M$ with a metric $\g$. At this stage, we do not impose any a priori relation between $\g$ and $\nabla$, and in particular, $\nabla$ is not the Riemannian connection corresponding to $\g$. Since $\{\a,\b\}$ is a frame field, the metric is fully determined by three real-valued functions on $\M$,
\[
\g_{\a\a} = \g(\a,\a) = |\a|^2,
\qquad
\g_{\b\b} = \g(\b,\b) = |\b|^2
\textand
\g_{\a\b} = \g(\a,\b).
\] 

For $p,q\in \M$, denote by $d(p,q)$ the distance between $p$ and $q$ induced by $\g$.
Note that it is not equal to the length of the $\nabla$-geodesic from $p$ to $q$. The following proposition bounds the discrepancy between the two:

\begin{proposition}
\label{prop:dist_estimate}
There exists constants $C>0$ and $\delta>0$, such that for all $p\in\M$ and every $v\in T_p\M$, 
$|v|<\delta$, 
\begin{equation}
\label{eq:dist_estimate}
|d(p,q) - |v|| \le C|v|^2,
\end{equation}
where $q = \expNAB_p v$.
\end{proposition}

\begin{proof}
First, since $\M$ is smooth and compact and $\nabla$ is smooth, the geodesic curvature of all $\nabla$-geodesics is bounded by some constant $K>0$. In such case, 
it was proved in \cite[Proposition 2.2]{KM16} that there exist constants $C_1>0$ and $\delta_1>0$, both depending on $K$, such that for every $\nabla$-geodesic $\gamma:I\to\M$ of length $\ell(\gamma)<\delta_1$,
\begin{equation}
0\le \ell(\gamma) - d(\gamma(0),\gamma(1)) \le C_1\, d^3(\gamma(0),\gamma(1)).
\label{eq:ineqA}
\end{equation}

Second, relying again on the compactness of $\M$, there exist constants $C_2>0$ and $\delta_2>0$, such that for every $p,q\in\M$ satisfying $d(p,q)<\delta_2$,
\begin{equation}
\sup_{|v|=1} ||\Pi_p^q v| - |v|| \le C_2\,d(p,q).
\label{eq:ineqB}
\end{equation}

Third, still by the compactness of $\M$ and the smoothness of $\nabla$, there exists a $\delta>0$, such that for all $p\in\M$ and $v\in T_p\M$ satisfying $|v|<\delta$, 
\[
\ell(\gamma) \le \min(\delta_1,\delta_2),
\] 
where $\gamma:I\to\M$ is given by $\gamma(t) = \expNAB_p(tv)$. Moreover,
\begin{equation}
\ell(\gamma) = \int_0^1 |\dot{\gamma}(t)|\, dt 
= \int_0^1 |(d\expNAB_p)_{v\, t}(v)|\, dt 
= \int_0^1 |\Pi_p^{\gamma(t)} v|\, dt,
\label{eq:ineqC}
\end{equation}
where the second equality follows from the chain rule and the third equality follows from Lemma~\ref{lemma:dexp}. Since for all $t\in I$, $d(p,\gamma(t)) \le \ell(\gamma) \le \delta_2$,
it follows from \eqref{eq:ineqB} that for all $t\in I$
\[
| |\Pi_p^{\gamma(t)} v| - |v|| \le C_2\, d(p,\gamma(t)\, |v| \le C_2\,\ell(\gamma)\, |v|.
\]
Substituting into \eqref{eq:ineqC},
\[
|\ell(\gamma) - |v|| \le \int_0^1 ||\Pi_p^{\gamma(t)} v| - |v||\, dt \le 
C_2\, \ell(\gamma) |v|.
\]
Iterating this last inequality,
\begin{equation}
\begin{split}
|\ell(\gamma) - |v|| &\le C_2\, |v|^2 + C_2 |\ell(\gamma) - |v|| \,|v| 
\le (C_2 + C_2^2 \delta_2) |v|^2.
\end{split}
\label{eq:ineqE}
\end{equation}
Denoting $q = \expNAB_p v$, we further obtain from \eqref{eq:ineqA} that
\begin{equation}
\begin{split}
|d(p,q) - \ell(\gamma)| &\le C_1\,d^3(p,q) \\
&\le C_1\delta_2 \,d^2(p,q) \\
&\le  C_1\delta_2 |v|^2 + C_1\delta_2 (d(p,q) + |v|)|d(p,q) - |v|| \\
&\le C_1\delta_2 |v|^2 + \eta\, |d(p,q) - |v||.
\end{split}
\label{eq:ineqF}
\end{equation}
where $\eta = C_1\delta_2 (\delta + \delta_2)$.
Combining \eqref{eq:ineqE} and \eqref{eq:ineqF}, 
\[
\begin{split}
|d(p,q) - |v|| &\le |d(p,q) - \ell(\gamma)| + |\ell(\gamma) - |v|| \\
&\le  (C_2 + C_2^2 \delta_2) |v|^2 +  C_1\delta_2 |v|^2 + \eta\, |d(p,q) - |v||,
\end{split}
\]
By taking $\delta_2$ and $\delta$ sufficiently small such that $\eta < 1$, we finally obtain that
\[
|d(p,q) - |v|| \le \frac{C_2 + C_2^2 \delta_2 + C_1\delta_2 }{1 -\eta} |v|^2,
\]
which completes the proof.
\end{proof}

\subsection{Triangulation}
\label{sec:triangulation}

For $\e>0$ sufficiently small, let $(V_\e,E_\e)$ be a graph of a hexagonal triangulation of a two-dimensional submanifold $\M_\e\subset \M$, that satisfies: 
\begin{enumerate}
\item $V_\e$ is $C\e$-dense in $\M$, with $C=2\max_{p\in\M}\{|a_p|,|b_p|, |c_p|\}$.
\item The $\nabla$-geodesic from a vertex to its neighbor has initial velocity equal to one of the crystallographic axes and continues for time $\e$.
\item The graph is maximal, in the sense that there is no $\M_\e'\subset \M$ with an associated graph $(V_\e',E_\e')$ that satisfies the above properties and contains $(V_\e,E_\e)$ as a strict subset.
\end{enumerate}

In more detail,
every point $p\in V_\e$, excluding boundary points, has six neighbors, given by
\[
\{\expNAB_p(\pm \e\a_p), \,\,\,
\expNAB_p(\pm \e\b_p), \,\,\,
\expNAB_p(\pm \e\c_p)\}.
\] 


Note that we have made here explicit use of Lemma~\ref{lemma:additive_exp} and the fact that $\a$, $\b$ and $\c$ are parallel vector fields to obtain that $\M$ can be triangulated in this way: indeed, for every $p\in\M$, a (local) triangulation of $T_p\M$ by straight lines parallel to $\a_p$, $\b_p$ and $\c_p$ maps under $\expNAB_p$ into a (local) triangulation of $\M$ with the desired properties.

We denote by $p\sim_\e q$ the fact that $p,q\in V_\e$ are neighbors in the graph.
Let $p\sim_\e q\sim_\e r\sim_\e p$ be the vertices of a triangle; as above, we denote by $\mathcal{T}_\nabla(p,q,r)\subset \M$ the convex hull of $p,q,r$ with respect to $\nabla$-geodesics. We denote by $\mathcal{K}_\e$ the collection of all $\nabla$-geodesic triangles with vertices in $V_\e$; we denote by $\mathcal{K}_\e^+\subset\mathcal{K}_\e$ those triangles that can be surrounded by a path going first along the positive $\a$ direction, then along the positive $\b$ direction, and finally along the positive $\c$ direction; we denote by $\mathcal{K}_\e^-\subset\mathcal{K}_\e$ those triangles that can be surrounded by a path going first along the negative $\a$ direction, then along the negative $\b$ direction, and finally along the negative $\c$ direction.
Every $K\in \mathcal{K}_\e^+$ is of the form
\[
K = \mathcal{T}_\nabla(p,q,r),
\] 
with the convention of ordering the vertices such that 
\[
q = \expNAB_p(\e\a_p),
\qquad
r = \expNAB_q(\e\b_q)
\textand
p = \expNAB_r(\e\c_r).
\]
Every $K\in \mathcal{K}_\e^-$ is of the form
\[
K = \mathcal{T}_\nabla(q,p,s),
\] 
with the convention of ordering the vertices such that 
\[
p = \expNAB_q(-\e\a_q),
\qquad
s = \expNAB_p(-\e\b_p)
\textand
q = \expNAB_s(-\e\c_s);
\]
see Figure~\ref{fig:triangles}.

\begin{figure}[h]
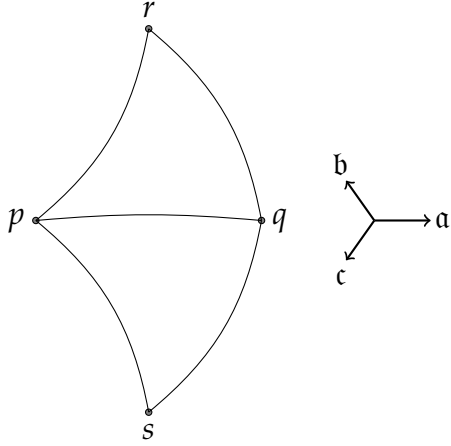

\begin{center}
\btkz[scale=1.5]
	\tkzDefPoint(0,0){p};
	\tkzDefPoint(2,0){q};
	\tkzDefPoint(1,-1.7){s};
	\tkzDefPoint(1,1.7){r};
	\tkzDrawPoints(p,q,r,s);
	\draw (0,0) edge[bend left=5] (2,0);
	\draw (0,0) edge[bend right=20] (1,1.7);
	\draw (0,0) edge[bend left=20] (1,-1.7);
	\draw (2,0) edge[bend left=20] (1,-1.7);
	\draw (2,0) edge[bend right=20] (1,1.7);
	\draw[->,thick] (3,0) -- (3.5,0);
	\draw[->,thick] (3,0) -- (2.75,0.353);
	\draw[->,thick] (3,0) -- (2.75,-0.353);
	\tkzText(3.6,0){$\a$};
	\tkzText(2.7,0.5){$\b$};
	\tkzText(2.7,-0.5){$\c$};
	\tkzLabelPoint[left](p){$p$}
	\tkzLabelPoint[right](q){$q$}
	\tkzLabelPoint[above](r){$r$}
	\tkzLabelPoint[below](s){$s$}
\etkz
\end{center}
\caption{$\nabla$-geodesic triangulation of $\M$ by triangles with edges parallel to the three crystallographic axes $\a$, $\b$ and $\c$.
The upper triangle $\mathcal{T}_\nabla(p,q,r)$ is in $\mathcal{K}_\e^+$; the lower $\mathcal{T}_\nabla(q,p,s)$ is in $\mathcal{K}_\e^-$.}
\label{fig:triangles}
\end{figure}

Finally, we denote by $\M_\e\subset \M$ the union of the geodesic triangles,
\[
\M_\e = \bigcup \mathcal{K}_\e.
\]
$\M_\e$ converges to $\M$ asymptotically, in the sense that there exists a constant $C>0$, such that
\begin{equation}
\Vol_\g(\M\setminus\M_\e) = \int_{\M\setminus\M_\e} \Volume \le C\e,
\label{eq:vol_delta_M_zero}
\end{equation}
where $\Volume$ is the Riemannian volume form.

\subsection{Discrete energies}
\label{sec:discrete_energies}

For every $\e>0$, the vertices $V_\e$ represent a discrete lattice. A configuration of that lattice is a map
\[
f_\e : V_\e \to \R^2.
\]
We denote the set of configurations by $L^2(V_\e;\R^2)$; the rationale for this notation will be made clear below.
With each configuration of $V_\e$ we associate a \emph{discrete elastic energy},
$\E_\e:L^2(V_\e;\R^2)\to\R$, having two contributions. 
The first is a bond energy 
\begin{equation}
\Ebond(f_\e) =   \sum_{p\sim_\e q} 
\mu_\e(p,q)\,
\Wb\brk{\frac{|f_\e(q) - f_\e(p)|}{d(p,q)}}.
\label{eq:Ebond}
\end{equation}
Here $\mu_\e(p,q)$ is the area of the set of points in $\M_\e$ that are closer to the edge between $p$ and $q$ than to any other edge; both distance and area are relative to $\g$.
The function
$\Wb:\R\to\R^+$ is an $\e$-independent bond energy, modeling pairwise inter-particle interactions.
Its argument is the relative elongation of an edge.

We assume that the bond energy $\Wb:[0,\infty)\to [0,\infty)$ satisfies the following conditions:
\begin{enumerate}
\setlength{\itemsep}{0pt}
\item $\Wb(r)=0$ if and only if $r=1$.
\item Coercivity: there exists a constant $\alPhi>0$ such that
\begin{equation}
\Wb(r)\ge \alPhi (r-1)^2.
\label{eq:coercive_Wb}
\end{equation}
\item Bounded growth: there exists a constant $\CPhi>0$ such that
\begin{equation}
\Wb(r) \le \CPhi(1 + r^2).
\label{eq:boundedness_Wb}
\end{equation}
\item Lipschitz continuity: there exists a constant $\LPhi>0$ such that for every $r,s>0$,
\begin{equation}
|\Wb(r) - \Wb(s)| \le \LPhi(1 + |r| + |s|)|r-s|.
\label{eq:lipschitz_Wb}
\end{equation}
\end{enumerate}
For example, the choice or $\Wb(r) = (r-1)^2$, corresponding to Hookean springs, satisfies these conditions.
As a non-example, an exponentiated Hencky energy \cite{NGL15}, $\Wb(r) = \sinh | \log r|$, is a bond energy not satisfying the bounded growth condition.

The second contribution to the discrete energy penalizes for changes in the (signed) volume of triangles,
\begin{equation}
\Evol(f_\e) = \sum_{\mathcal{T}_\nabla(p,q,r)\in\mathcal{K}_\e} \mu_\e(p,q,r)\,
\Psi\brk{\frac{(f_\e(q) - f_\e(p))\wedge(f_\e(r) - f_\e(q))}{\e^2 \, \nu(p,q,r)  \, (\partial_1\wedge \partial_2)}}.
\label{eq:Evol}
\end{equation}
Here, $\mu_\e(p,q,r)$ is the area of $\mathcal{T}_\nabla(p,q,r)$, and $\partial_1\wedge \partial_2$ is the standard unit bivector in $\R^2$; the argument of $\Psi$ is a ratio of top-rank multivectors in $\R^2$, hence can be viewed as a scalar. 
$\nu(p,q,r)$ is defined as follows:
set
\begin{equation}
\nu = |\a\wedge\b| = (\g_{\a\a} \g_{\b\b} - \g_{\g\b}^2)^{1/2} : \M\to \R,
\label{eq:nu}
\end{equation}
and denote by $\nu(p,q,r)$ the value of $\nu$ at the center of mass of $\mathcal{T}_\nabla(p,q,r)$.

We assume that the volumetric function $\Psi:\R\to [0,\infty)$ satisfies the following conditions:
\begin{enumerate}
\setlength{\itemsep}{0pt}
\item $\Psi(a)=0$ if and only if $a=1$.
\item Coercivity: there exists a constant $\alPsi>0$ such that
\begin{equation}
\Psi(a)>\alPsi \sqrt{|a|} \quad \text{ for all } \,a<0.
\label{eq:coercive_Psi}
\end{equation}
\item Bounded growth: there exists a constant $\CPsi>0$ such that
\begin{equation}
\Psi(a)<\CPsi(1+|a|).
\label{eq:boundedness_Psi}
\end{equation}
\item Lipschitz continuity: there exists a constant $\LPsi$ such that for every $a,b$,
\begin{equation}
|\Psi(a)-\Psi(b)|\le \LPsi |a-b|.
\label{eq:lipschitz_Psi}
\end{equation}
\end{enumerate}
For example, $\Psi(a) = \beta|a-1|$, satisfies these conditions.

The total discrete elastic energy is the sum of bond energy and the volumetric energy,
\begin{equation}
\E_\e(f_\e) = \Ebond(f_\e) + \Evol(f_\e).
\label{eq:Etotal}
\end{equation}

\subsection{Piecewise-affine extension}

\subsubsection{Extension of discrete configurations}

Given $f_\e\in L^2(V_\e;\R^2)$, we extend it into a function $F_\e\in W^{1,\infty}(\M,\R^2)$; we denote the extension map $f_\e \mapsto F_\e$ by $\iota_\e$. 

Within $\M_\e$, we extend $f_\e$ in each triangle as described in Section~\ref{sec:prelim}:
for $\mathcal{T}_\nabla(p,q,r) \in\mathcal{K}_\e^\pm$, 
\[
\begin{split}
F_\e(\expNAB_p(\pm s\e a_p \mp t \e \c_p)) &=  (1-s-t) f_\e(p) + s\, f_\e(q) + t\,f_\e(r),
\end{split}
\]
where $0\le s,t,s+t\le 1$.
By Proposition~\ref{prop:dF2}, within a triangle $\mathcal{T}_\nabla(p,q,r) \in\mathcal{K}_\e^\pm$, 
\begin{equation}
\begin{aligned}
dF_\e(\pm\e \a) &= f_\e(q) - f_\e(p) \\
dF_\e(\pm\e \b) &= f_\e(r) - f_\e(q) \\
dF_\e(\pm\e \c) &= f_\e(p) - f_\e(r).
\end{aligned}
\label{eq:dF_in_triangles}
\end{equation}

We extend $F_\e$ to $\M\setminus \M_\e$ such that
\begin{equation}
\label{eq:F_e_extension_to_boundary}
\|dF_\e\|_{L^\infty(\M)} \le \Cext \|dF_\e|_{\M_\e}\|_{L^\infty(\M_\e)}
\textand 
\|dF_\e\|_{L^2(\M)} \le \Cext \|dF_\e|_{\M_\e}\|_{L^2(\M_\e)},
\end{equation}
where $\Cext>0$ is independent of $\e$ and $F_\e$.

Such an extension can be achieved, for example,
as follows: extend $\M$ slightly into a larger manifold $\tilde{\M}$, such that $\M\Subset \tilde{\M}$, and extend $\g$ and $\nabla$ smoothly to $\tilde{\M}$.
For $\e$ small enough, we can extend $(V_\e,E_\e)$ into a geodesic triangulation $(\tilde{V}_\e,\tilde{E}_\e)$ of $(\tilde{\M}_\e,\tilde{\g})$, where $\M\subset \tilde{\M}_\e\subset \tilde{\M}$.
Every vertex in $v\in \tilde{V}_\e\setminus V_\e$ has at least one neighbor in $V_\e$  (and at most three).
Denote these vertices by $v_1,\ldots,v_k$.
Given $f_\e$, we extend it to $v$ by defining $f_\e(v) = \frac{1}{k} \sum_i f_\e(v_i)$.
We then extend $f_\e$ into a piecewise-affine function $F_\e:\tilde{\M}_\e\to \R^d$ as above, and restrict it to $\M$.
A straightforward calculation, using the uniform bounds on the angles between edges, shows that conditions \eqref{eq:F_e_extension_to_boundary} are satisfied.

\subsubsection{An integral representation of the discrete energy}
We start be defining seven families of piecewise-constant functions $\M_\e\to\R$, parametrized by the lattice spacing $\e$. The first three,
\[
\rho_\e^{\a},\rho_\e^{\b},\rho_\e^{\b}: \M_\e\to\R 
\]
return for $x\in\mathcal{T}_\nabla(p,q,r)$ the relative area in $\mathcal{T}_\nabla(p,q,r)$ of the region that is closest to the edge parallel to $\a$, $\b$ and $\c$, respectively; in a Euclidean equilateral triangle, these functions would be equal to $1/3$.
For two triangles $\mathcal{T}_\nabla(p,q,r)$ and $\mathcal{T}_\nabla(q,p,s)$ sharing the edge $pq$, we have
\[
\mu_\e(p,q) = \int_{\mathcal{T}_\nabla(p,q,r) \cup \mathcal{T}_\nabla(q,p,s)} \rho_\e^\a\,\Volume.
\]

The next three families of functions,
\[
D_\e^{\a},D_\e^{\b},D_\e^{\c}: \M_\e\to\R,
\]
return for a point in $\mathcal{T}_\nabla(p,q,r)$ 
the $\e$-rescaled distances between the vertices of that triangle:  
\[
D_\e^{\a} = \frac{d(q,p)}{\e},
\qquad
D_\e^{\b} = \frac{d(r,q)}{\e}
\Textand
D_\e^{\c} = \frac{d(p,r)}{\e}.
\]

Let $F_\e = \iota_\e(f_\e)$. By \eqref{eq:dF_in_triangles} and the definition of $D_\e^{\a,\b,\c}$, for every point in $\mathcal{T}_\nabla(p,q,r)$,
\[
\begin{aligned}
\frac{dF_\e(\a)}{D_\e^{\a}} &=\pm \frac{f_\e(q) - f_\e(p)}{d(p,q)} \\
\frac{dF_\e(\b)}{D_\e^{\b}} &= \pm\frac{f_\e(r) - f_\e(q)}{d(q,r)} \\
\frac{dF_\e(\c)}{D_\e^{\c}} &= \pm\frac{f_\e(p) - f_\e(r)}{d(r,p)},
\end{aligned}
\]
where the sign depends on whether $\mathcal{T}_\nabla(p,q,r)$ is a triangle in $\mathcal{K}_\e^+$ or $\mathcal{K}_\e^-$.

Finally, we define 
\[
\nu_\e:\M_\e\to\R, 
\]
which for $x\in\mathcal{T}_\nabla(p,q,r)$ returns $\nu = |\a\wedge \b|$ evaluated at the center of mass, i.e., the value of $\nu(p,q,r)$ as defined in \eqref{eq:nu}.

Consider next the determinant of $dF_\e$. The intrinsic expression for the determinant is
\[
\det dF_\e = \star^2_\euc \circ (dF_\e \wedge dF_\e) \circ \star_\g^0 (1),
\]
where $\star_\g^0:\Omega^0(\M) \to \Omega^2(\M)$ and $\star^2_\euc:\Omega^2(\R^2)\to\Omega^0(\R^2)$ are the Hodge-dual operators on the graded algebras of multivectors (with respect to the metric $\g$ on $\M$ and the Euclidean metric in $\R^2$).
Now,
\[
\star_\g^0 (1) = \frac{\a\wedge\b}{|\a\wedge\b|} = \frac{\a\wedge\b}{\nu},
\]
hence
\[
(dF_\e \wedge dF_\e) \circ \star_\g^0 (1) = \frac{dF_\e(\a)\wedge dF_\e(\b)}{\nu},
\]
and in $\mathcal{T}_\nabla(p,q,r)$,
\[
\begin{split}
\det dF_\e &= \frac{dF_\e(\a)\wedge dF_\e(\b)}{\nu\, (\partial_1 \wedge\partial_2)} 
= \frac{\nu_\e}{\nu} \frac{(f_\e(q) - f_\e(p))\wedge (f_\e(r) - f_\e(q))}{\e^2 \nu(p,q,r) \, (\partial_1 \wedge\partial_2)} ,
\end{split}
\]
where we used \eqref{eq:dF_in_triangles} again, and the fact that $\nu_\e=\nu(p,q,r)$ in $\mathcal{T}_\nabla(p,q,r)$.

With these notations, we can write the discrete bond energy of $f_\e$ in terms of its extension $F_\e = \iota_\e(f_\e)$ as an integral over $\M_\e$:
\begin{equation}
\Ebond(f_\e) = \int_{\M_\e} \Wbond(dF_\e)\,\Volume,
\label{eq:Ebond_cont}
\end{equation}
where $\Wbond:T^*\M\otimes\R^2\to\R$ is given by
\begin{equation}
\Wbond(A) = \rho_\e^{\a} \, \Wb\brk{\frac{|A(\a)|}{D_\e^{\a}}} +
\rho_\e^{\b}\,  \Wb\brk{\frac{|A(\b)|}{D_\e^{\b}}} +
\rho_\e^{\c}\,  \Wb\brk{\frac{|A(\c)|}{D_\e^{\c}}}.
\label{eq:Weps}
\end{equation}
The discrete volumetric energy of $f_\e$ can be written as follows:
\begin{equation}
\Evol(f_\e) = \int_{\M_\e} \Wvol(dF_\e)\,\Volume,
\label{eq:Evol_cont}
\end{equation}
where $\Wvol:T^*\M\otimes\R^2\to\R$ is given by
\begin{equation}
\Wvol(A) = \Psi \brk{\frac{\nu}{\nu_\e}  \det A}.
\label{eq:Wvol}
\end{equation}
Thus, the total discrete energy is given by
\begin{equation}
\E_\e(f_\e) = \int_{\M_\e} \Wtot(dF_\e)\,\Volume,
\end{equation}
where $\Wtot:T^*\M\otimes\R^2\to\R$ is given by
\begin{equation}
\Wtot(A) = \Wbond(A) + \Wvol(A).
\label{eq:Wtot}
\end{equation}

We end this section by establishing asymptotic properties of the piecewise-constant functions defined on the triangulated surfaces:

\begin{lemma}
\label{lm:limit_distances}
We have the following uniform limits in $\M_\e$: for $\fraku=\a,\b,\c$,
\[
\lim_{\e\to0} D_\e^{\fraku} = |\fraku|  \equiv D^\fraku.
\]
\end{lemma}

\begin{proof}
This is an immediate consequence of Proposition~\ref{prop:dist_estimate}.
\end{proof}

\begin{lemma}
\label{lm:limit_relative_weight}
We have the following uniform limits in $\M_\e$: for $\fraku=\a,\b,\c$,
\[
\lim_{\e\to0} \rho_\e^{\fraku} = \frac{|\fraku|}{|\a| + |\b| + |\c|} \equiv \rho^\fraku.
\]
\end{lemma}

\begin{proof}
Let $x\in \mathcal{T}_\nabla(p,q,r)$. 
The triangle $\mathcal{T}_\nabla(p,q,r)$ is the image under $\expNAB_p$ of $\mathcal{T}_p(0,\e\a_p,\e \b_p)$. The functions $\rho^\a$, $\rho^\b$ and $\rho^\c$ evaluated at $p$ are the relative areas in $(\mathcal{T}_p(0,\e\a_p,\e \b_p),\g_p)$ of the regions that are closest to the edges parallel to $\a_p$, $\b_p$ and $\c_p$, respectively.
The statement follows from the fact that
the restriction $\expNAB_p: \mathcal{T}_p(0,\e\a_p,\e \b_p)\to \mathcal{T}_\nabla(p,q,r)$ satisfies $|d\expNAB_p- \id|<C\e$, where $C$ is independent of $p$.
\end{proof}

\begin{lemma}
We have the following uniform limit in $\M_\e$:
\[
\lim_{\e\to0} \frac{\nu}{\nu_\e} = 1.
\]
\end{lemma}

\begin{proof}
This follows from the smoothness of $\g$ and the compactness of $\M$.
\end{proof}

\section{$\Gamma$-convergence}
\label{sec:Gamma}

In this section we prove our main result:

\begin{theorem}[$\Gamma$-convergence]
The sequence of discrete energies
\label{thm:Gamma_convergence}
\[
\mathcal{E}_\e : L^2(V_\e;\R^2) \to \R
\]
$\Gamma$-converge with respect to the $L^2$-norm (defined below) to the limit functional $\mathcal{F} : L^2(\M;\R^2)\to\R\cup\{\infty\}$ defined by
\begin{equation}
\mathcal{F}(F) = \Cases{
\int_\M QW(dF)\, \Volume & F\in W^{1,2}(\M;\R^2) \\
\infty & \text{otherwise},
}
\label{eq:calF}
\end{equation}
where $QW$ is the quasi-convex envelope of $W: T^*\M\otimes\R^2\to\R$, given by
\begin{equation}
W(A) = \rho^\a\, \Wb\brk{\frac{|A(\a)|}{|\a|}} + \rho^\b\, \Wb\brk{\frac{|A(\b)|}{|\b|}} + 
\rho^\c\, \Wb\brk{\frac{|A(\c)|}{|\c|}} + \Psi(\det A).
\label{eq:W}
\end{equation}
\end{theorem}

In this setting, the quasi-convex envelope is defined fiberwise: for each $p\in \M$, $QW|_p$ is the largest quasiconvex function on $T_p^*\M\otimes\R^2$, smaller than $W|_p$;
see also Section~\ref{sec:prop_of_limit}.

We furthermore have the standard convergence of minimizers for equi-coercive $\Gamma$-converging functionals: 
\begin{proposition}[Convergence of minimizers]
\label{prop:compactness}
Let $f_{\e}\in L^2(V_{\e};\R^2)$, $\e\to0$,  be a sequence of approximate minimizers of $\mathcal{E}_{\e}$\, i.e., $\mathcal{E}_{\e}(f_{\e})-\inf \mathcal{E}_{\e}\to 0$.
Then, there exists a sequence $c_\e\in \R^2$, such that the sequence $\iota_{\e}(f_{\e})-c_\e$ is compact in $L^2(\M;\R^2)$; all its limit points are minimizers of $\mathcal{F}$.
Moreover
\[
\lim_{\e\to 0} \inf_{L^2_\e(\M_\e;\R^2)} \mathcal{E}_\e = \min_{L^2(\M;\R^2)} \mathcal{F}.
\]
\end{proposition}

In order to characterize the convergence of functionals, we first need to specify a topology for its domain. The appropriate topology in the present case is $L^2(\M;\R^2)$, where $L^2(V_\e;\R^2)$ is embedded in $L^2(\M;\R^2)$ via the extension map $\iota_\e$. The first step is to extend $\mathcal{E}_\e: L^2(V_\e;\R^2) \to \R$ into a functional $I_\e : L^2(\M;\R^2)\to\R\cup\{\infty\}$ defined as follows,
\[
I_\e(F) = \Cases{\mathcal{E}_\e(\iota_\e^{-1}(F)) & F\in L^2_\e(\M;\R^2) \\
\infty & \text{otherwise,}
}
\]
where $L^2_\e(\M;\R^2)$ is the image of $L^2(V_\e;\R^2)$ under the extension map $\iota_\e$.

By the sequential compactness property for separable spaces \cite[Theorem~8.5]{Dal93}, for every sequence $\e_n\to0$, the sequence of functionals $I_{\e_n}$ has a $\Gamma$-converging subsequence. By the Urysohn property \cite[Proposition~8.3]{Dal93}, if there exists a functional $\mathcal{F}:L^2(\M;\R^2) \to\R\cup\{\infty\}$, which is the $\Gamma$-limit of every converging subsequence, then $\mathcal{F}$ is the $\Gamma$-limit of $I_\e$ as $\e\to0$. 

We will prove that $\mathcal{F}$ given by \eqref{eq:calF} is the $\Gamma$-limit of every $\Gamma$-converging subsequence, hence the $\Gamma$-limit of $I_\e$. Denote by $I$ the limit of a (not relabeled) subsequence $I_\e$ (we also omit the index $n$ for ease of notation). We will prove that $I = \mathcal{F}$ by showing first that $I\le \mathcal{F}$ and then that $I\ge \mathcal{F}$;  we will treat separately the case where $\mathcal{F}$ assumes an infinite value, in which case it suffices to show that $I\ge\mathcal{F}$.

\subsection{Piecewise-affine approximation}

We start by showing that every function in $W^{1,2}(\M;\R^2)$ can be approximated by functions in $L^2_\e(\M;\R^2)$. This property is necessary in order to construct recovery sequences.

\begin{proposition}
\label{prop:approx}
Let $F\in W^{1,2}(\M;\R^2)$. Then, there exists a sequence of functions $F_\e\in L^2_\e(\M;\R^2)$, with $\e\to0$, such that
\[
F_\e \to F
\qquad \text{in $W^{1,2}(\M;\R^2)$}.
\]
\end{proposition}

\begin{proof}
First, it suffices to prove the claim for real-valued functions. 
Second, by a standard density argument, it suffices to prove the claim for $C^\infty(\M)$ functions. 

So let $F\in C^\infty(\M)$
and set $F_\e = \iota_\e(F|_{V_\e}) \in L^2_\e(\M)$. We will prove that
\[
F_\e \to F
\qquad \text{in $W^{1,2}(\M)$}.
\]
First, by the construction of the triangulations there exists 
a constant $C>0$ such that the vertices $V_\e$ form a $C\e$-dense subset of $\M$. 
Therefore, $F_\e=F$ over a $C\e$-dense set.
Second, denote by $L$ the Lipschitz constant of $F$; by \eqref{eq:F_e_extension_to_boundary}, the Lipschitz constant of $F_\e$ is bounded by $\Cext L$, hence $\|F_\e-F\|_\infty < (L+\Cext L) C \e$.

Next, we show that 
\[
\|dF_\e-dF\|_{L^\infty(\mathcal{T}_\nabla(p,q,r))}=O(\e)
\]
in each triangle $\mathcal{T}_\nabla(p,q,r)$.
To this end,
we start by comparing $dF_p$ and $(dF_\e)_p$. 
By \eqref{eq:dF_in_triangles},
\[
(dF_\e)_p(\e\a) = F(q)-F(p) \Textand (dF_\e)_p(\e\c)=F(p)-F(r).
\]
By Taylor's expansion,
\[
F(q) = F(p) + dF_p(\e\a) + O(\e^2), \qquad F(r) = F(p) - dF_p(\e\c) + O(\e^2).
\]
It follows that
\[
|dF_p(\a) - (dF_\e)_p(\a)|= O(\e) \Textand |dF_p(\c) - (dF_\e)_p(\c)| = O(\e).
\]

We then note that $dF(\a), dF(\c)$ are Lipschitz maps (since $F$ is smooth), and by \eqref{eq:dF_in_triangles}, $dF_\e(\a)$ and $dF_\e(\c)$ are constants in $\mathcal{T}_\nabla(p,q,r)$. 
It follows that that for any $x\in \mathcal{T}_\nabla(p,q,r)$ 
\[
|dF_x(\a) - (dF_\e)_x(\a)| \le |dF_x(\a) - dF_p(\a)| + |dF_p(\a) - (dF_\e)_p(\a)| = O(\e),
\]
and similarly for $\a$ replaced by $\c$.
Since, by the compactness of $\M$, the angle between $\a$ and $\c$ is uniformly bounded away from $0$ and $\pi$, and since $|\a|,|\c|$ are uniformly bounded away from zero, we obtain that $|dF-dF_\e| = O(\e)$ in $\M_\e$.

Finally, since by \eqref{eq:vol_delta_M_zero} $\Vol_\g(\M\setminus\M_\e)=O(\e)$ , and since $F_\e$ are uniformly Lipschitz, it follows that $\|dF_\e-dF\|_{L^p(\M)}=O(\e)$ for every $p<\infty$, which completes the proof.
\end{proof}

\subsection{Properties of $\Wtot$ and $W$}
We proceed to establish some properties of the energy densities $\Wtot$ and $W$, which are required in the subsequent analysis; these follow from the assumed properties of the discrete bond and volumetric energy functions $\Wb$ and $\Psi$.

The metrics $\g$ and $\euc$ induce the standard Frobenius metric on the vector bundle $T^*\M\otimes\R^2$. We denote by $\O{\g,\euc}$ the sub-bundle of $T^*\M\otimes\R^2$  of isometries $(T\M,\g)\to(\R^2,\euc)$; we  denote by $\SO{\g,\euc}$ the sub-bundle of orientation-preserving isometries.

\begin{proposition}
\label{pn:properties_of_W}
The functions $\Wtot,W:T^*\M\otimes\R^2\to \R$ given by \eqref{eq:Wtot} and \eqref{eq:W} satisfy the following properties:
\begin{enumerate}
\setlength{\itemsep}{0pt}
\item $W(A)=0$ if and only if $A\in\SO{\g,\euc}$.
\item Coercivity of $W$: there exists a constant $\alW>0$ such that 
\begin{equation}
\begin{aligned}
W(A) &\ge \alW\,\dist^2(A,\SO{\g,\euc}),
\end{aligned}
\label{eq:coercivW}
\end{equation}
where the distance is with respect to the norm induced by $\g$ and $\euc$.

\item Bounded growth of $W$: there exists a constant $\CW>0$, such that 
\begin{equation}
W(A) \le \CW(1 + |A|^2).
\label{eq:bddW}
\end{equation}

\item Uniform coercivity of $\Wtot$: there exists a function $h:(0,1)\to\R$ satisfying $h(\e)\to0$ as $\e\to0$, such that 
\begin{equation}
\begin{aligned}
\Wtot(A) &\ge \alW\,\dist^2(A,\SO{\g,\euc}) - h(\e) \,(1+ |A|^2).
\end{aligned}
\label{eq:unif_coerciv}
\end{equation}

\item Uniformly bounded growth of $\Wtot$: 
\begin{equation}
\Wtot(A) \le \CW(1 + |A|^2).
\label{eq:unif_bdd}
\end{equation}

\item Let $F_\e$ be a family of functions uniformly bounded in $W^{1,2}(\M;\R^2)$. Then,
\begin{equation}
\lim_{\e\to0} \int_{\M_\e} |\Wtot(dF_\e) - W(dF_\e)|\,\Volume = 0.
\label{eq:WetoW}
\end{equation}
\end{enumerate}
\end{proposition}

\begin{proof}\
\begin{enumerate}
\setlength{\itemsep}{0pt}
\item
By the properties of $\Wb$ and $\Psi$, $W(A)$ vanishes if and only if
\[
|A(\a)| = |\a|, 
\qquad
|A(\b)| = |\b|,
\qquad
|A(\c)| = |\c|, 
\]
and
\[
\det A=1.
\]
Since $\a/|\a|$, $\b/|\b|$ and  $\c/|\c|$ are unit-vector fields, it follows that $A$ preserves the lengths of two independent vectors and their sum. It follows from elementary linear algebra that $A$ is a (local) isometry (see Lemma~\ref{lem:3indep}), i.e., $A\in\O{\g,\euc}$. The positivity of the determinant implies that $A\in\SO{\g,\euc}$.

\item
By the coercivity \eqref{eq:coercive_Wb} of $\Wb$ and the coercivity \eqref{eq:coercive_Psi} of $\Psi$,
\[
\begin{split}
W(A) &\ge \alPhi \min(\rho^\a,\rho^\b,\rho^\c)\, \sum_{\fraku = \a,\b,\c} \brk{\Abs{A\brk{\frac{\fraku}{|\fraku|}}}-1}^2 
+ \alPsi|\det A|^{1/2}\, \ind_{\det A<0}.
\end{split}
\]
Since, from the compactness of $\M$, the angle between $\a$ and $\b$ is bounded away from zero and $\pi$, and the ratio between their lengths is bounded and bounded away from $0$, it follows from Lemma~\ref{lem:tech_lemma_new} that there exists a constant $c>0$, such that
\[
\sum_{\fraku = \a,\b,\c} \brk{\Abs{A\brk{\frac{\fraku}{|\fraku|}}}-1}^2 \ge c \dist^2(A,\O{\g,\euc}.
\]
hence
\[
W(A) \ge \alPhi c \min(\rho^\a,\rho^\b,\rho^\c) \dist^2(A,\O{\g,\euc}) + \alPsi |\det A|^{1/2}\, \ind_{\det A<0},
\]
Finally, using Lemma~\ref{lem:tech_lemma_3}, we deduce that there exists a constant $\alW>0$, such that
\[
W(A) \ge \alW\,\dist^2(A,\SO{\g,\euc}).
\]

\item
By the boundedness \eqref{eq:boundedness_Wb}, \eqref{eq:boundedness_Psi} of $\Wb$ and $\Psi$,
\[
\begin{split}
W(A) &\le \CPhi \brk{
\rho^\a  \frac{|A(\a)|^2}{|\a|^2}  +
\rho^\b  \frac{|A(\b)|^2}{|\b|^2} 
+ \rho^\c  \frac{|A(\c)|^2}{|\c|^2} + 1}  + \CPsi(1 + |\det A |) \\
&\le (\CPhi+\CPsi) \,\brk{|A|^2 + 1} .
\end{split}
\]
where we used the fact that $\rho^\a+ \rho^\b+ \rho^\c =1$, the fact that the operator norm is bounded from above by the Frobenius norm, and the inequality,
\begin{equation}
|\det A|  \le |A|^2 ,
\label{eq:det_ineq}
\end{equation}
which holds in two dimensions.

\item
Combining \eqref{eq:Weps} and \eqref{eq:Wvol},
\[
\begin{split}
\Wtot(A) &= W(A) + \sum_{\fraku = \a,\b,\c}  
\brk{\rho_\e^{\fraku} \, \Wb\brk{\frac{|A(\fraku)|}{D_\e^{\fraku}}} - \rho^{\fraku} \, \Wb\brk{\frac{|A(\fraku)|}{D^{\fraku}}}} \\
&\quad +
\brk{\Psi\brk{\frac{\nu}{\nu_\e}  \det A} - \Psi\brk{\det A}}
\end{split}
\]
By the boundedness properties \eqref{eq:boundedness_Wb} of $\Wb$, the Lipschitz  continuity \eqref{eq:lipschitz_Wb} of $\Wb$ and the Lipschitz  continuity \eqref{eq:lipschitz_Psi} of $\Psi$,
\[
\begin{split}
|\Wtot(A) - W(A)| &\le 
\CW  \sum_{\fraku=\a,\b,\c} |\rho_\e^\fraku - \rho^\fraku| \brk{1 + \frac{|A(\fraku)|^2}{(D^\fraku)^2}} \\
&\quad+ \LPhi \sum_{\fraku = \a,\b,\c} \rho_\e^{\fraku} \, 
\brk{1 + \frac{|A(\fraku)|}{D_\e^{\fraku}} + \frac{|A(\fraku)|}{D^{\fraku}}}\Abs{\frac{1}{D_\e^{\fraku}} - \frac{1}{D^{\fraku}}}|A(\fraku)| \\
&\quad +
\LPsi \Abs{\frac{\nu}{\nu_\e}   - 1} |\det A|.
\end{split}
\]

The three terms of the right-hand side can be bounded by using the fact that $|\det A| \le |A|^2$ and that $\rho_\e^\fraku - \rho^\fraku$, $1/D^\fraku - 1/D_\e^{\fraku}$ and $\nu/\nu_\e-1$ converge to $0$ uniformly, yielding
\begin{equation}
|\Wtot(A) - W(A)| \le  h(\e)\, (1+|A|^2),
\label{eq:Wtot-W}
\end{equation}
where $h(\e)$ tends to zero as $\e\to0$. In particular,
\[
\Wtot(A) \ge \alW\,\dist^2(A,\SO{\g,\euc}) - h(\e)\, (1+|A|^2),
\]

\item
The uniform boundedness property of $\Wtot$ follows from the bound \eqref{eq:Wtot-W} on $|\Wtot(A) - W(A)|$ and the boundedness \eqref{eq:bddW} of $W$, possibly having to enlarge $\CW$.

\item 
We have
\[
|\Wtot(dF_\e) - W(dF_\e) |\le  h(\e)\, (1+|dF_\e|^2),
\]
and limit \eqref{eq:WetoW} follows from $h(\e)\to0$ and the uniform 
boundedness of $dF_\e$ in $L^2(\M;\R^2)$.
\end{enumerate}
\end{proof}

\subsection{Proof of Theorem~\ref{thm:Gamma_convergence}}


\begin{proposition}[Infinite case]
Let $I:L^2(\M;\R^2)\to\R\cup\{\infty\}$ be a $\Gamma$-limit of a sequence $I_\e$, $\e\to0$. If
$F\in L^2(\M;\R^2)\setminus W^{1,2}(\M;\R^2)$ then
\[
I(F) = \infty = \mathcal{F}(F).
\]
\end{proposition}

\begin{proof}
Let $F_\e$ be a recovery sequence; namely, $F_\e\to F$ in $L^2(\M;\R^2)$ and
\[
I(F) = \lim_{\e\to0} I_\e(F_\e).
\]
Suppose, by contradiction, that $I(F)<\infty$. This implies that the sequence $I_\e(F_\e)$ is eventually bounded, by say, $C$. 
Without loss of generality, we may assume that all $F_\e$ are in $L^2_\e(\M;\R^2)$. By the uniform coercivity \eqref{eq:unif_coerciv} of $\Wtot$, 
\begin{equation}
\label{eq:coercivity_implies_boundedness}
\begin{split}
C &\ge I_\e(F_\e) \\
&= \int_{\M_\e} \Wtot(dF_\e)\, \Volume \\
&\ge \alW \int_{\M_\e} \dist^2(dF_\e,\SO{\g,\euc})\,\Volume  - h(\e) \int_{\M_\e} (1 + |dF_\e|^2)\,\Volume \\
&\ge \frac{\alW}{2} \int_{\M_\e} (|dF_\e|^2 - 4)\,\Volume - h(\e) \int_{\M_\e} (1 + |dF_\e|^2)\,\Volume \\
&\ge \frac{1}{\Cext}\brk{\frac{\alW}{2} - h(\e)} \|dF_\e\|^2_{L^2(\M)} -
\brk{2\alW - h(\e)} \, \Vol_\g(\M_\e).
\end{split}
\end{equation}
In the passage to the fourth line we used the inequality $|a-b|^2 \ge |a|^2/2 - |b|^2$ with $a=dF_\e$ and $b\in\SO{\g,\euc}$;
in the passage to the last line we used \eqref{eq:F_e_extension_to_boundary}.
Since $h(\e)\to0$ and $\Vol_\g(\M_\e) \to \Vol_\g(\M)$,
it follows that $dF_\e$ is uniformly bounded in $L^2(\M; T^*\M\otimes \R^2)$; since $F_\e$ converges in $L^2(\M;\R^2)$, it is uniformly bounded in $W^{1,2}(\M;\R^2)$, hence has a weakly converging subsequence. By the uniqueness of the limit, this limit is $F$. 
Hence, $F\in  W^{1,2}(\M;\R^2)$ --- a contradiction.
\end{proof}

\begin{proposition}[Finite case: lower bound]
Let $I:L^2(\M;\R^2)\to\R\cup\{\infty\}$ be a $\Gamma$-limit of a sequence $I_\e$, $\e\to0$.
Then, for all $F\in W^{1,2}(\M;\R^2)$,
\[
I(F) \ge \mathcal{F}(F).
\]
\end{proposition}

\begin{proof}
Let $F\in W^{1,2}(\M;\R^2)$.
If $I(F) = \infty$ then the claim is trivial. Otherwise,
let $F_\e\to F$ be a recovery sequence, namely, $F_\e\to F$ in $L^2(\M;\R^2)$ and 
\[
I(F) = \lim_{\e\to0} I_\e(F_\e) < \infty.
\]
Without loss of generality, we may assume that $F_\e\in L^2_\e(\M;\R^2)$ for all $\e$, hence
\[
I(F) = \lim_{\e\to0} \,\,\int_{\M_\e} \Wtot(dF_\e)\, \Volume < \infty.
\]
By the coercivity of $\Wtot$ (as above), $dF_\e$ is bounded in $L^2(\M;T^*\M\otimes\R^2)$.
Since $F_\e$ also converges in $L^2(\M;\R^2)$, it has a subsequence that weakly converges in $W^{1,2}(\M;\R^2)$; by the uniqueness of the limit, this subsequence converges to $F$. Then,
for every $\tilde{\M}\subset\M$ satisfying $\dist(\tilde{\M},\partial\M)>0$,
\[
\begin{split}
I(F) &= \lim_{\e\to0} \,\,\int_{\M_\e} \Wtot(dF_\e)\, \Volume \\
&\ge \limsup_{\e\to0} \,\,\int_{\M_\e} W(dF_\e)\, \Volume - \liminf_{\e\to0} \,\,\int_{\M_\e} |\Wtot(dF_\e) - W(dF_\e)|\, \Volume \\
&\ge \limsup_{\e\to0} \,\,\int_{\tilde{\M}} W(dF_\e)\, \Volume  
\ge  \limsup_{\e\to0} \,\,\int_{\tilde{\M}} QW(dF_\e)\, \Volume \\
&\ge  \liminf_{\e\to0} \,\,\int_{\tilde{\M}} QW(dF_\e)\, \Volume 
\ge  \int_{\tilde{\M}} QW(dF)\, \Volume.
\end{split}
\]
The  to the second line follows from the triangle inequality; 
the next inequality follows from restricting the domain of integration to $\tilde{\M}$ and from \eqref{eq:WetoW};
the next inequality follows from the fact that any function is greater or equal to its quasi-convex envelope; 
the next inequality is trivial;
the last inequality follows from the fact that an integral functional is lower-semicontinuous if (and only if) the integrand is quasi-convex.
The proof is complete by taking the supremum over $\tilde{\M}$.
\end{proof}

\begin{proposition}[Finite case: upper bound]
Let $I:L^2(\M;\R^2)\to\R\cup\{\infty\}$ be a $\Gamma$-limit of a sequence $I_\e$, $\e\to0$.
For all $F\in W^{1,2}(\M;\R^2)$,
\[
I(F) \le \mathcal{F}(F),
\]
\end{proposition}

\begin{proof}
Given $F\in W^{1,2}(\M;\R^2)$, apply Proposition~\ref{prop:approx} to construct a sequence $F_\e\in L^2_\e(\M;\R^2)$ converging to $F$ strongly in $W^{1,2}(\M;\R^2)$. By the lower-semicontinuity property of $\Gamma$-limits,
\[
\begin{split}
I(F) &\le \liminf_{\e\to0} I_\e(F_\e) \\
&=  \liminf_{\e\to0}  \int_{\M_\e} \Wtot(dF_\e)\, \Volume \\
&\le \int_{\M} W(dF)\,\Volume +  \limsup_{\e\to0}  \int_{\M_\e} |\Wtot(dF_\e) - W(dF_\e)|\, \Volume \\
&\quad + \limsup_{\e\to0}  \int_{\M_\e} |W(dF_\e) - W(dF)|\, \Volume  \\
&=  
 \int_{\M} W(dF)\, \Volume,
\end{split}
\]
where the passage to the last line follows from \eqref{eq:WetoW} and the Lipschitz property of $W$ (which follows from the Lipschitz property of $\Wb$ \eqref{eq:lipschitz_Wb} and $\Psi$ \eqref{eq:lipschitz_Psi}).

We would be done if we could replace $W$ on the right-hand side by its quasi-convex envelope. Since $QW\le W$, this replacement cannot be performed directly.

Denote the right-hand side by $J(F)$, and denote by $\tilde{J}$ the extension of $J$ to a functional on $L^2(\M;\R^2)$, by defining it to be infinite on the rest of the domain.  
Since $I$ is a $\Gamma$-limit, it is sequentially lower-semicontinuous with respect to the strong $L^2(\M;\R^2)$ topology, hence
$I(F) \le \Gamma \tilde{J}(F)$,
where $\Gamma: \tilde{J}\mapsto \Gamma\tilde{J}$ denotes the lower-semicontinuous envelope.  Denote also by $\Gamma_w$ the sequential lower-semicontinuous envelope with respect to the weak $W^{1,2}$-topology. By \cite[Lemma 5]{LR95},
$\Gamma \tilde{J} = \widetilde{\Gamma_w J}$, 
and for $F\in W^{1,2}(\M;\R^2)$,
\[
I(F) \le \widetilde{\Gamma_w J}(F) = \Gamma_w J(F).
\]
By \cite{AF84} (see also \cite[Thoerem 3.2]{KM14} for a statement in non-Euclidean settings),
\[
I(F) \le \Gamma_w \int_\M W(dF)\, \Volume = \int_\M QW(dF)\, \Volume = \mathcal{F}(F),
\]
which completes the proof.
\end{proof}

\subsection{Proof of Proposition~\ref{prop:compactness}}

The growth condition of $QW$ (which follows from that of $W$) implies that $\inf \mathcal{E}_{\e}$ is a bounded sequence.
We take a (non-relabeled) subsequence $\e$ such that $\inf \mathcal{E}_{\e}$ converges.
Let $f_{\e}$ be a sequence of approximate minimizers, 
i.e.,
$\brk{\mathcal{E}_{\e}(f_\e) - \inf \mathcal{E}_{\e}} \to 0$.
Set
$F_{\e}=\iota_{\e}(f_{\e})$ and $c_\e = \int_{\M} F_{\e}\,\Volume$. Since the functionals $I_\e$ are translation-invariant, we can assume without loss of generality that $c_\e=0$.
By the same calculation as in \eqref{eq:coercivity_implies_boundedness}, it follows that $dF_{\e}$ is bounded in $L^2(\M;T^*\M\otimes\R^2)$.
By the Poincar\'e inequality, $F_{\e}$ converges weakly in $W^{1,2}(\M;\R^2)$ (modulo a subsequence) to some function $F$.

Let $\tilde{f}_\e\in L^2(V_{\e};\R^2)$ be a recovery sequence of some $\tilde{F}\in W^{1,2}(\M;\R^2)$. Then, by the definition of $\Gamma$-convergence,
\[
\mathcal{F}(F) \le \liminf I_{\e}(F_{\e}) = \lim (\inf \mathcal{E}_{\e}) \le \lim \mathcal{E}_{\e}(\tilde{f}_\e) = \mathcal{F}(\tilde{F}).
\]
Since $\tilde{F}$ is arbitrary, $F$ is a minimizer of $\mathcal{F}$; by choosing $\tilde{F}=F$, we obtain
\[
\mathcal{F}(F) = \lim (\inf \mathcal{E}_{\e_n}).
\]

\section{Properties of the limit}
\label{sec:prop_of_limit}

In this section we analyze the limiting functional $\mathcal{F}$, whose properties are determined by its integrand $QW$. 
Since $W\ge 0$, on each fiber $T^*_p\M\otimes\R^2$, the quasiconvex envelope is given by
\begin{equation}
QW(A) = \inf\BRK{
\frac{\int_{D} W(A+d\phi\circ\kappa)\, \omega}{\int_{D}  \omega} ~:~ \phi\in C^{\infty}_0(D;\R^2)}.
\label{eq:QCdef}
\end{equation}
Here, $D\subset T_p\M$ is the closed unit disc and $\omega$ is an arbitrary volume form on $T_p\M$.  
The bundle map $\kappa :  T_p\M\times T_p\M \to TT_p\M$
is the canonical identification of the tangent bundle of a vector space, so that for $\xi\in D$,
\[
(d\phi \circ \kappa)_\xi  : T_p\M \to \R^2.
\]
These coordinate-free definitions (see \cite{KM14}) reduce to the well-known Euclidean formulations of the quasi-convex envelope by choosing a basis to $T_p\M$; the formula \eqref{eq:QCdef} is well-known, see \cite[Theorem~6.9]{Dac08}. 

\begin{proposition}[Properties of the limit functional]\
\begin{enumerate}
\item Frame indifference:  for $R\in\O{\g,\euc}$,
\[
QW(RA) = QW(A).
\]
\item Rigidity: There exists an $\alpha>0$, such that for all $A\in T^*\M\otimes\R^2$,
\begin{equation}
\label{eq:bounds_QW}
\alpha \dist^2(A, \SO{\g,\euc})\le QW(A) \le W(A).
\end{equation}
In particular, Property \eqref{eq:energy_density_basic_assumption} holds:
\[
QW(A) = 0 \qquad\text{if and only if}\qquad A\in \SO{\g,\euc}.
\]
\item No stress-free configuration: if $\g$ is not flat, then
\[
\min_{F\in W^{1,2}(\M;\R^2)} \mathcal{F}(F) > 0.
\]
\end{enumerate}
\label{pn:properties_limit_functional}
\end{proposition}

\begin{proof}
The frame indifference 
of $QW$ follow from the frame indifference 
of $W$ and formula \eqref{eq:QCdef}.
$QW\le W$ follows from the definition of $QW$; for the lower bound in \eqref{eq:bounds_QW},
note that by \eqref{eq:coercivW},
\[
W(A) \ge \alW\,\dist^2(A,\SO{\g,\euc}),
\]
hence
\[
QW(A) \ge \alW\,Q\dist^2(A,\SO{\g,\euc}).
\]
Denote $\tilde{W}(A) = \dist^2(A,\SO{\g,\euc})$. 
We need to show that $Q\tilde{W} \ge c\tilde{W}$ for some $c>0$.
Let $A\in T_p^*\M\otimes \R^2$, and let $\phi\in C^\infty_0(D;\R^2)$, where $D\subset T_p\M$ is the closed unit disc.
Using \eqref{eq:QCdef} we have
\[
Q\tilde{W}(A) \ge \int_D \dist^2\brk{A+d\phi\circ \kappa, \SO{\g_p,\euc}}\omega,
\]
where we chose $\omega$ such that $\int_D \omega = 1$.
The rigidity theorem
\cite[Theorem~3.1]{FJM02b} implies the existence of $c>0$, independent of $A$ and $\phi$, and a rigid map $R\in \SO{\g_p,\euc}$ such that
\[
\int_D \dist^2\brk{A+d\phi\circ \kappa, \SO{\g_p,\euc}}\omega \ge c\int_D \Abs{A+d\phi\circ \kappa - R}^2 \omega.
\]
Since $\phi$ is compactly supported,
\[
\begin{split}
\int_D \Abs{A+d\phi\circ \kappa - R}^2 \omega &= \int_D \brk{|A-R|^2 + 2\g_p(A-R,d\phi\circ \kappa) + |d\phi\circ \kappa|^2 } \omega \\
 &= \int_D \brk{|A-R|^2  + |d\phi\circ \kappa|^2}  \omega 
\ge \int_D |A-R|^2   \omega \ge \tilde{W}(A).
\end{split}
\]
Combining these inequalities we obtain $Q\tilde{W} \ge c\tilde{W}$.

Finally, by Proposition~\ref{prop:compactness}, the minimum of $\mathcal{F}$ is obtained; denote the minimizing function by $F\in W^{1,2}(\M;\R^2)$.
If $\mathcal{F}(F)=0$, then it follows from the above argument that $dF\in \SO{\g,\euc}$ almost everywhere.
It follows by \cite[Lemma~3.1]{LP10} (see also \cite{KMS17}) that $F$ is smooth, hence $dF\in \SO{\g,\euc}$ everywhere and therefore $\g$ is flat.
\end{proof}

\begin{remark}
An alternative proof of the second part can be obtained using the explicit formula for $Q\dist^2(A,\SO{\g,\euc})$ calculated in \cite[Example 4.2]{Sil01}.
The proof above, however, readily generalizes to higher dimensions.
\end{remark}

\subsection{The conformal case}
\label{sec:conformal}

So far, no relation between the metric $\g$ and the connection $\nabla$ has been assumed  Thus, there is no reason to expect any sort of internal symmetry of the limit functional.
In many cases, e.g., when an initially Euclidean body undergoes an inhomogeneous, yet isotropic expansion, the metric and the connection are related---the angles between the original lattice directions are preserved. 
This is the case considered in this section; as we show below, such an assumption results in additional structure of the limit functional:
\begin{enumerate}
\item The limit functional admits a \emph{material connection}, in the sense of \cite[p.~66]{Wan67}. 
A material connection is an affine-connection on $\M$, such that the energy density is invariant with respect to its parallel transport.
This is a generalized form of homogeneity in Euclidean bodies, namely, independence on spatial coordinates, which is equivalent to invariance under the Euclidean parallel transport.

In the present case, the material connection is neither $\nabla$ nor the Levi-Civita connection of $\g$, but rather a connection which is metrically consistent with $\g$ and has the same geodesics as $\nabla$.

\item In a special case, the limit functional admits a discrete right symmetry (isotropy).
\end{enumerate}

\begin{definition}
The metric $\g$ is said to be \emph{conformal} with respect to $\nabla$, if there exists a positive scalar function $\phi:\M\to\R$, such that for every pair of $\nabla$-parallel vector fields, $\xi,\eta\in\calX(\M)$ and for every $p,q\in\M$
\begin{equation}
\frac{\g_p(\xi_p,\eta_p)}{\phi^2(p)} = \frac{\g_{q}(\xi_{q},\eta_{q})}{\phi^2(q)}.
\label{eq:conformal}
\end{equation}
\end{definition}

It is easy to see that $\g$ is conformal with respect to $\nabla$ if and only if it satisfies \eqref{eq:conformal} for  $\xi,\eta\in\{\a,\b\}$. Moreover, the conformal factor is only determined up to a multiplicative constant.

\begin{proposition}
\label{pn:constant_angles}
For a conformal metric, the angles between the parallel vector fields $\a$, $\b$ and $\c$ are constant.
\end{proposition}

\begin{proof}
This is immediate from the definition of conformality.
\end{proof}

\begin{proposition}
Eq.~\eqref{eq:conformal} holds if and only if the connection given by
\[
\tilde{\nabla}_X Y = \nabla_X Y + X(\sigma) Y,
\]
where $\sigma = \log \phi$, is flat and metric.
This connection is uniquely defined as the metric connection with torsion 
\[
T(X,Y) = X(\sigma)Y - Y(\sigma)X.
\]
\end{proposition}

\begin{proof}
If $X$ is a $\nabla$-parallel vector field, then $X/\phi$ is $\tilde{\nabla}$-parallel,
\[
\tilde{\nabla}_Y \brk{\frac{X}{\phi}} 
= \nabla_Y \brk{\frac{X}{\phi}} +Y(\sigma) \frac{X}{\phi} 
= \frac{1}{\phi} \nabla_Y X - \frac{1}{\phi^2} Y(\phi) X + Y(\sigma) \frac{X}{\phi} 
= 0.
\]
Hence $(\a/\phi,\b/\phi)$ is a $\tilde{\nabla}$-parallel frame field, hence $\tilde{\nabla}$ is flat.
$\tilde{\nabla}$ is metric if and only if the lengths of $\a/\phi$ and $\b/\phi$ and the angle between them are constant, which holds if and only if \eqref{eq:conformal} holds.

Finally, the torsion of $\tilde{\nabla}$ is given by
\[
\begin{split}
T(X,Y) &= \tilde{\nabla}_XY - \tilde{\nabla}_YX - [X,Y] \\
	&= \nabla_XY + X(\sigma)Y - \nabla_YX - Y(\sigma)X - [X,Y] \\
	&=X(\sigma)Y - Y(\sigma)X.
\end{split}
\]
Since for every antisymmetric $(2,1)$-tensor there exists a unique metric connection whose torsion is the given tensor (this is similar to the proof of uniqueness of the Levi-Civita connection), this torsion characterizes $\tilde{\nabla}$ uniquely.
\end{proof}

\begin{proposition}[Existence of a material connection]
\label{pn:material_connection}
Denote by $\tilde{\Pi}_p^q$ the parallel transport operator induced by $\tilde{\nabla}$. 
Then,
$\tilde{\Pi}^*W = W$, i.e., for every $A\in T^*_p\M\otimes \R^2$,
\[
W_p(A) = W_q(A\circ \tilde{\Pi}_q^p),
\]
and similarly for $QW$.
\end{proposition}

\begin{proof}
First, by the previous proposition, $\a/|\a|$, $\b/|\b|$ and $\c/|\c|$ are $\tilde{\nabla}$-parallel.
Second, $\rho^\a$, $\rho^\b$ and $\rho^\c$ defined in Lemma~\ref{lm:limit_relative_weight} are constant functions. Third, since the connection is metric, $\det(A\circ\tilde{\Pi}) = \det A$.
Therefore, for any $A\in T^*_p\M\otimes \R^2$,
\[
\begin{split}
W_q(A\circ \tilde{\Pi}_q^p) 
	&= \sum_{\fraku=\a,\b,\c} \rho^\fraku\, \Wb\brk{\frac{|A(\tilde{\Pi}_q^p\fraku(q))|}{|\fraku(q)|}} + \Psi(\det (A\tilde{\Pi}_q^p)) \\
	&=\sum_{\fraku=\a,\b,\c} \rho^\fraku\, \Wb\brk{\frac{|A(\fraku(p))|}{|\fraku(p)|}} + \Psi(\det A) 
	= W_p(A).
\end{split}
\]
By \eqref{eq:QCdef}, this property is inherited by $QW$.
\end{proof}

\begin{proposition}[Discrete right-symmetry]
Let $\g$ be conformal. Suppose that there exists a point $o$ where $|\a_o| = |\b_o|$ and
$\angle(\a_o,\b_o) = 2\pi/3$.
Then $QW$ is right-invariant with respect to $\pi/3$ rotations. 
\label{pn:right_symmetry}
\end{proposition}

\begin{proof}
First, note that 
\[
|\c_o|^2 = |\a_o + \b_o|^2  = |\a_o|^2 + 2|\a_o| |\b_o|\cos(2\pi/3)  + |\b_o|^2 = |\a_o|^2. 
\]
Therefore,
\[
\cos(\angle(\c_o,\a_o)) = \frac{\g_o(\c_o,\a_o)}{|\a_0|^2} =  \frac{-|\a_0|^2 -\g_o(\b_o,\a_o)}{|\a_0|^2} = -\frac{1}{2},
\]
hence also $\angle(\c_o,\a_o) = 2\pi/3$, and similarly for $\angle(\b_o,\c_o)$.
Since the angles between the lattice axes are constant, the angles between $\a$, $\b$ and $\c$ are $2\pi/3$ at all points.

Next, by definition of the conformal factor, for every $p\in\M$,
\[
\frac{|a_p|^2}{\phi(p)} = \frac{|a_o|^2}{\phi(o)} = 
\frac{|b_o|^2}{\phi(o)} =\frac{|b_p|^2}{\phi(p)},
\]
hence $|\a_p| = |\b_p|$ and similarly for $|\c_p|$. 
Therefore, $\rho^\a = \rho^\b = \rho^\c = 1/3$, and
\begin{equation}
\label{eq:W_conformal}
W(A) = \frac13\BRK{\Wb\brk{\frac{|A(\a)|}{\phi}} + \Wb\brk{\frac{|A(\b)|}{\phi}} + 
\Wb\brk{\frac{|A(\c)|}{\phi}}} + \Psi(\det A).
\end{equation}

Since rotations by an angle of $\pi/3$  in $T_p\M$ amount to a relabeling of the vectors $\pm\a$, $\pm\b$ and $\pm\c$, it follows from \eqref{eq:W_conformal} that $W$ is invariant under such rotations. 
By \eqref{eq:QCdef}, this property is inherited by $QW$.
\end{proof}

\section{Discussion: extensions to other models}
\label{sec:dicussion}

This paper is concerned with obtaining a model of incompatible elasticity as a $\Gamma$-limit of discrete particle models in two dimensions. 
To this end, an ``incompatible elastic model" is a model that satisfies the properties of Proposition~\ref{pn:properties_limit_functional}.
The discrete lattice models are constructed by using a flat, symmetric connection $\nabla$ to obtain an hexagonal discretization of the manifold.
In this section we discuss several possible variations and extensions of the discrete models and their limits.

\paragraph{Higher dimensions}
All the results in this paper are readily generalizable to higher dimensions (we restricted our analysis to two dimensions for the sake of clarity).
For a $d$-dimensional Riemannian manifold $(\M,\g)$ endowed with a flat, symmetric connection $\nabla$, we can choose $\nabla$-parallel frame fields $v_1,\ldots,v_d$, and define crystallographic axes by
\[
\{a_i\}_{i=1}^{2^d-1} := \BRK{ \sum_{j\in A} a_j ~:~ A\subset \{1,\ldots,d\}, A\ne \emptyset}
\]
We may then construct a triangulation of $\M$ by $d$-simplices whose edges are the crystallographic axes as in Section~\ref{sec:triangulation}, and define discrete bond energies and discrete volumetric energies as in Section~\ref{sec:discrete_energies}.
The only required modification is raising the right-hand sides of \eqref{eq:coercive_Psi}--\eqref{eq:lipschitz_Psi} by a power of $2/d$.
The rest of the analysis remains virtually unchanged, and the limit energy satisfies Proposition~\ref{pn:properties_limit_functional}.

\paragraph{Other lattice structures}
In this paper, we considered discrete models based on hexagonal lattices.
Similar results (i.e., yielding in the limit an incompatible elastic model) can be obtained for other lattices, provided that the discrete energies $\Wtot$ satisfy lower and upper bounds similar to the ones in Proposition~\ref{pn:properties_of_W}.
For example, in a cubic lattice, an incompatible elastic model can be obtained
only if the volumetric energy (or, alternatively, an energy term related to angular deviations) penalizes shear deformations sufficiently; indeed, pairwise bond energy in a cubic lattice is indifferent to shear.

\paragraph{Avoiding interpenetration}
The discrete volumetric energy $\Evol$ considered in this paper penalizes for orientation-reversing \eqref{eq:coercive_Psi}, in a manner ensuring that the coercivity estimates \eqref{eq:coercivW} and \eqref{eq:unif_coerciv} hold, while not changing significantly the upper bounds \eqref{eq:boundedness_Psi}.
Although penalizing orientation-reversing is physically sound, a more physical approach would be to completely rule out interpenetration, for example, by defining $\Evol$ with a volumetric function $\Psi$ that satisfies 
\[
\Psi(a) = \infty \,\,\,\,\,
\forall a\le 0
\Textand \lim_{x\to 0} \Psi(x) = \infty.
\]
Such a function $\Psi$ violates the bound \eqref{eq:boundedness_Psi}; obtaining a $\Gamma$-limit from such discrete models that prevent interpenetration is beyond the scope of this paper.

Another approach for avoiding interpenetration in the limit, which is not as physically sound, but, yet, can be adapted to our case, is the following (see a similar approach in \cite{ACG11}): 
consider a sequence of volumetric functions $\Psi_k$, satisfying
\[
\lim_{k\to \infty} \inf_{a\le 0} \Psi_k(a) = \infty,
\]
and take an iterated $\Gamma$-limit for the energies, first $\e\to 0$ and then $k\to \infty$.

\appendix
\section{Technical lemmas}

\begin{lemma}
\label{lem:3indep}
Let $V$ and $W$ be two-dimensional inner-product spaces. Let $A\in \Hom(V,W)$. If there exist two independent vectors $x,y\in V$ such that
\[
|A(x)| = |x|,
\qquad
|A(y)| = |y|
\Textand
|A(x+y)| = |x+y|,
\]
then $A$ is an isometry.
\end{lemma}

\begin{proof}
It follows from the polarization identity that $(A(x),A(y)) = (x,y)$, and therefore $A$ preserves the inner-product, hence it is an isometry.
\end{proof}

\begin{lemma}
\label{lem:tech_lemma_1}
Let $V$ and $W$ be two-dimensional inner-product spaces. Let $x,y\in V$ be independent vectors of equal length. Then, for every $A\in \Hom(V,W)$,
\[
|A|^2 \le \frac{2}{1-\cos\theta}\brk{\frac{|Ax|^2}{|x|^2}  + \frac{|Ay|^2}{|y|^2} + \frac{|A(x+y)|^2}{|x+y|^2}},
\]
where $\theta$ is the angle between $x$ and $y$.
\end{lemma}

\begin{proof}
It suffices to prove the lemma for unit vectors.
Denote $c = \cos\theta = (x,y)$.
The vectors
$x$  and $(y - c\, x)/\sqrt{1-c^2}$
are orthonormal, hence
\[
|A|^2 = |Ax|^2 + \frac{|A(y - c\, x)|^2}{1-c^2} = \frac{|Ax|^2 + |Ay|^2 - 2c(Ax,Ay)}{1-c^2}.
\]
Now,
\[
\begin{split}
|A|^2 &= \frac{2}{1-c} \cdot \frac{1}{2(1+c)} \brk{|Ax|^2 + |Ay|^2 - 2c(Ax,Ay)} \\
&= \frac{2}{1-c}  \brk{|Ax|^2 + |Ay|^2 + \frac{|A(x+y)|^2 - 2(1+c)(|Ax|^2 + |Ay|^2 + (Ax,Ay))}{|x+y|^2}} \\
&\le \frac{2}{1-c}  \brk{|Ax|^2 + |Ay|^2 + \frac{|A(x+y)|^2}{|x+y|^2}},
\end{split}
\]
where in the last step we used the fact that
$|Ax|^2 + |Ay|^2 + (Ax,Ay) > 0$.
\end{proof}

\begin{lemma}
\label{lem:tech_lemma_2}
Let $V$ and $W$ be two-dimensional inner-product spaces. Let $x,y\in V$ be two independent vectors. Then, there exists a constant $C$ depending continuously on the angle $\theta$ between $x$ and $y$ and the ratio of their lengths $r = |y|/|x|$, such that for every $A\in \Hom(V,W)$,
\[
|A|^2 \le C\brk{\frac{|Ax|^2}{|x|^2}  + \frac{|Ay|^2}{|y|^2} + \frac{|A(x+y)|^2}{|x+y|^2}}.
\]
\end{lemma}

\begin{proof}
Without loss of generality we can assume that $|y| > |x|$ 
(otherwise Lemma~\ref{lem:tech_lemma_1} applies).
Set 
\[
v=x+\alpha y
\Textand
w=(1-\alpha) y,
\]
where
\[
\alpha = \frac{r^2-1}{2r(r +\cos \theta)}
\]
is chosen such that $|v| = |w|$. 
Also, $v+w = x+y$. 
Note that $\alpha \in \brk{(r-1)/2r, (r+1)/2r}\subset (0,1)$, and in particular, $v$ and $w$ are independent.
The angle between $v$ and $w$ depends only on $\alpha$ and $\theta$, and therefore on $r$ and $\theta$.
By the previous lemma, there exists a $C = C(r,\theta)$ such that
\[
\begin{split}
|A|^2 &\le C\brk{\frac{|Av|^2}{|v|^2}  + \frac{|Aw|^2}{|w|^2} + \frac{|A(v+w)|^2}{|v+w|^2}} \\
	& = C\brk{\frac{|Ax|^2 + \alpha^2|Ay|^2 + 2\alpha(Ax,Ay)}{(1-\alpha)^2|y|^2}  + \frac{|Ay|^2}{|y|^2} + \frac{|A(x+y)|^2}{|x+y|^2}} \\
	& \le C\brk{\frac{|Ax|^2 + \alpha^2|Ay|^2 + \alpha(|Ax|^2 + |Ay|^2)}{(1-\alpha)^2|y|^2}  + \frac{|Ay|^2}{|y|^2} + \frac{|A(x+y)|^2}{|x+y|^2}} \\
	& = C\brk{\frac{1+\alpha}{(1-\alpha)^2r^2}\frac{|Ax|^2}{|x|^2}  + \brk{1+\frac{\alpha^2 +\alpha}{(1-\alpha)^2}}\frac{|Ay|^2}{|y|^2} + \frac{|A(x+y)|^2}{|x+y|^2}} \\
	& \le C'\brk{\frac{|Ax|^2}{|x|^2}  + \frac{|Ay|^2}{|y|^2} + \frac{|A(x+y)|^2}{|x+y|^2}},
\end{split}
\]
where in the passage to the third line we used the inequality $2ab\le a^2 + b^2$.
\end{proof}

\begin{lemma}
\label{lem:tech_lemma_new}
Let $V$ and $W$ be two-dimensional inner-product spaces, and let $x,y\in V$ be two independent vectors.
Then there exists a constant $C>0$ depending continuously on the angle $\theta$ between $x$ and $y$ and the ratio of their lengths $r = |y|/|x|$, such that for every $A\in \Hom(V,W)$,
\begin{equation}
\label{ineq}
\dist^2(A,\O{2}) \le C\Brk{\brk{\frac{|Ax|}{|x|}-1}^2  + \brk{\frac{|Ay|}{|y|}-1}^2 + \brk{\frac{|A(x+y)|}{|x+y|}-1}^2}
\end{equation}
\end{lemma}

\begin{proof}
For ease of notation, we will write $\fraku_1 = x$, $\fraku_2 = y$ and $\fraku_3 = x+y$.
We will also use the notation $f(A)\lesssim g(A)$, meaning that there exists a constant $C>0$, such that $f(A)\le C\, g(A)$ for all $A\in\Hom(V,W)$, or for a subset of $\Hom(V,W)$ as specified.
 
First, we show that \eqref{ineq} holds for every $|A|$ large enough: for $|A|>\sqrt{2}$, 
\[
\dist^2(A,\O{V,W}) \le \left(|A| + \sqrt{2}\right)^2 < 4|A|^2.
\]
By Lemma~\ref{lem:tech_lemma_2}, 
\[
|A|^2 \lesssim \sum_{i=1}^3 \frac{|A\fraku_i|^2}{|\fraku_i|^2} \le 3 \max_i \frac{|A\fraku_i|^2}{|\fraku_i|^2}.
\]
It follows that for $|A|$ large enough, we also have
\[
\max_i \brk{\frac{|A\fraku_i|}{|\fraku_i|} -1 }^2 \ge \frac{1}{2} \max_i \frac{|A\fraku_i|^2}{|\fraku_i|^2}.
\]
Using all of the above, we obtain
\[
\dist^2(A,\O{V,W}) \lesssim \max_i \frac{|A\fraku_i|^2}{|\fraku_i|^2} \lesssim \max_i \brk{\frac{|A\fraku_i|}{|\fraku_i|} -1 }^2 \le \sum_i \brk{\frac{|A\fraku_i|}{|\fraku_i|} -1 }^2,
\]
which proves \eqref{ineq} for large enough $|A|$.

Next, we note that it suffices to prove that
\begin{equation}
\label{ineq2}
|P-I|^2 \lesssim \sum_i \brk{\frac{|P\fraku_i|}{|\fraku_i|} -1 }^2
\end{equation}
for any symmetric, semi-positive definite $P$, in the vicinity of $I$.
Indeed, if $A=UP$ is the polar decomposition of $A$, then $\dist(A,\O{V,W}) = |P-I|$, and the right-hand side of \eqref{ineq} is left-$\O{W}$-invariance as a function of $A$.
Since we have already showed that \eqref{ineq2} holds for large enough $|P|$, we need to prove it in a compact ball, which means it is enough to prove it in the vicinity of the zero-set of the right-hand side of \eqref{ineq2}, which the identity matrix (Lemma~\ref{lem:3indep}).
Indeed, suppose that \eqref{ineq2} holds in $B_{r}(I)$ for some $r>0$ with constant $C_1$. Consider the continuous function $|P-I|^{-2}\sum_i \brk{\frac{|P\fraku_i|}{|\fraku_i|} -1 }^2$ in the compact set $\overline{B_{R}(I)}\setminus B_{r}(I)$. This function attains a non-zero minimum $C_2$, hence  \eqref{ineq2} holds in $B_{R}(I)$ with constant $\max(C_1,C_2^{-1})$.

Writing $P = I + \e B$, where $B$ is a symmetric matrix with $|B|=1$, we therefore need to show that for small $\e>0$,
\[
\e^2 \lesssim \sum_i \brk{\frac{|\fraku_i+ \e B\fraku_i|}{|\fraku_i|} -1 }^2.
\]
Taylor expanding, 
\[
\begin{split}
\sum_i \brk{\frac{|\fraku_i+ \e B\fraku_i|}{|\fraku_i|} -1 }^2 
	&= \sum_i \brk{\sqrt{1 + 2\e \frac{\fraku_i}{|\fraku_i|}\cdot B \frac{\fraku_i}{|\fraku_i|} + \e^2 \frac{|B\fraku_i|}{|\fraku_i|}} -1 }^2 \\
	&= \e^2\sum_i \brk{\frac{\fraku_i}{|\fraku_i|}\cdot B \frac{\fraku_i}{|\fraku_i|}}^2 + O(\e^3).
\end{split}
\]
In order to complete the proof of \eqref{ineq2}, we need to show that $\sum_i \brk{\fraku_i\cdot B \fraku_i}^2$ cannot vanish for $|B|=1$.
Assuming otherwise---that $\fraku_i\cdot B \fraku_i=0$ for $i=1,2,3$,
\[
0=(x+y)\cdot B(x+y) = x\cdot Bx + 2y\cdot Bx + y \cdot By = 2y\cdot Bx,
\]
hence $Bx$ is perpendicular to $x$ and $y$, i.e., $Bx=0$. Similarly $By=0$, hence $B=0$, in contradiction.

Finally, note that the constant $C$ in \eqref{ineq} depends only on $r$ and $\theta$ since the claim is invariant under (simultaneous) rotation and rescaling of $x$ and $y$.
\end{proof}

\begin{lemma}
\label{lem:tech_lemma_3}
Let $V$ and $W$ be $d$-dimensional inner-product spaces. Then, for every $A\in \Hom(V,W)$,
\begin{equation}
\label{eq:ineq_SO_O_det}
\dist^2(A,\SO{V,W}) \le \dist^2(A,\O{V,W}) + 4|\det A|^{1/d} \ind_{\BRK{\det A<0}}.
\end{equation}
\end{lemma}

\begin{proof}
Let $\sigma_1\ge \sigma_2\ge\ldots\ge \sigma_d\ge 0$ be the singular values of $A$.
Then
\[
\dist^2(A,\O{V,W}) = \sum_{i=1}^{d} (\sigma_i - 1)^2 \textand |\det A| = \prod_{i=1}^{d} \sigma_i
\]
If $\det A\ge 0$, then
\[
\dist^2(A,\SO{V,W}) = \sum_{i=1}^{d} (\sigma_i - 1)^2 = \dist^2(A,\O{V,W}),
\]
which shows the equality in \eqref{eq:ineq_SO_O_det} in this case.
If $\det A<0$, then
\[
\begin{split}
\dist^2(A,\SO{V,W}) &= \sum_{i=1}^{d-1} (\sigma_i - 1)^2 + (\sigma_d + 1)^2 \\
&= \dist^2(A,\O{V,W}) + 4\sigma_d 
\le \dist^2(A,\O{V,W}) + 4|\det A|^{1/d}.
\end{split}
\]
\end{proof}

\bibliographystyle{amsalpha}
\providecommand{\bysame}{\leavevmode\hbox to3em{\hrulefill}\thinspace}
\providecommand{\MR}{\relax\ifhmode\unskip\space\fi MR }
\providecommand{\MRhref}[2]{%
  \href{http://www.ams.org/mathscinet-getitem?mr=#1}{#2}
}
\providecommand{\href}[2]{#2}
\footnotesize{}

\end{document}